\documentclass[12pt]{article}
\usepackage{fullpage, amsmath, amsthm,amsfonts,url,amssymb,stmaryrd, mathrsfs}

\usepackage[pdftex]{color}
\usepackage[all]{xy}

\newtheorem{theorem}{Theorem}[subsection]
\newtheorem{theoremstar}{Theorem}
\newtheorem{lemma}[theorem]{Lemma}

\newtheorem*{conjstar}{Conjecture}
\newtheorem{prop}[theorem]{Proposition}

\theoremstyle{definition}
\newtheorem{defn}[theorem]{Definition}

\newtheorem{example}[theorem]{Example}
\newtheorem{caution}[theorem]{Caution}
\newtheorem{remark}[theorem]{Remark}

\newtheorem{notation}[theorem]{Notation}
\newtheorem{formulary}[theorem]{Formulary}

\numberwithin{equation}{subsection}

\newcommand{\CC}{\mathbf{C}}

\newcommand{\QQ}{\mathbf{Q}}

\newcommand{\RR}{\mathbf{R}}
\newcommand{\ZZ}{\mathbf{Z}}
\newcommand{\bfA}{\mathbf{A}}

\newcommand{\bbD}{\mathbf{D}}
\newcommand{\bbL}{\mathbf{L}}

\newcommand{\calC}{\mathcal{C}}

\newcommand{\calL}{\mathcal{L}}

\newcommand{\calO}{\mathcal{O}}
\newcommand{\calR}{\mathcal{R}}
\newcommand{\calW}{\mathcal{W}}

\newcommand{\gothm}{\mathfrak{m}}

\newcommand{\gothp}{\mathfrak{p}}

\newcommand{\scrD}{\mathscr{D}}

\newcommand{\scrM}{\mathscr{M}}
\newcommand{\scrS}{\mathscr{S}}
\newcommand{\scrT}{\mathscr{T}}
\newcommand{\scrV}{\mathscr{V}}
\newcommand{\scrX}{\mathscr{X}}

\newcommand{\an}{\mathrm{an}}
\newcommand{\alg}{\mathrm{alg}}

\newcommand{\bs}{\backslash}
\newcommand{\cn}{\colon}

\newcommand{\crys}{\mathrm{crys}}

\newcommand{\dif}{\mathrm{dif}}

\newcommand{\dR}{\mathrm{dR}}

\newcommand{\Ga}{\Gamma}
\newcommand{\ga}{\gamma}
\newcommand{\id}{\mathrm{id}}
\newcommand{\Ind}{\mathrm{Ind}}

\newcommand{\om}{\omega}
\newcommand{\ord}{\mathrm{ord}}
\newcommand{\ov}[1]{{\overline{#1}}}
\newcommand{\pair}{\mathrm{pair}}
\newcommand{\perf}{\mathrm{perf}}
\newcommand{\pst}{\mathrm{pst}}
\def\Qp{\QQ_p}
\def\Qell{\QQ_\ell}

\newcommand{\rig}{\mathrm{rig}}

\newcommand{\unr}{\mathrm{unr}}

\newcommand{\spp}{\mathrm{sp}}
\newcommand{\st}{\mathrm{st}}

\newcommand{\wh}[1]{{\widehat{#1}}}

\newcommand{\wt}[1]{{\widetilde{#1}}}
\newcommand{\Zp}{\ZZ_p}
\newcommand{\Lotimes}{\mathop{\stackrel{\bbL}{\otimes}}}
\newcommand{\vep}{\varepsilon}
\newcommand{\vphi}{\varphi}

\DeclareMathOperator{\Aut}{Aut}
\DeclareMathOperator{\Cone}{Cone}

\DeclareMathOperator{\coker}{coker}
\DeclareMathOperator{\End}{End}
\DeclareMathOperator{\Fib}{Fib}
\DeclareMathOperator{\Fil}{Fil}

\DeclareMathOperator{\Frac}{Frac}
\DeclareMathOperator{\Gal}{Gal}
\DeclareMathOperator{\Gr}{Gr}

\DeclareMathOperator{\img}{img}
\DeclareMathOperator{\rank}{rank}
\DeclareMathOperator{\Max}{Max}

\DeclareMathOperator{\Spec}{Spec}

\begin{document}

\title{On the parity conjecture in finite-slope families}
\author{Jonathan Pottharst and Liang Xiao}

\maketitle
\begin{abstract}
We generalize to the finite-slope setting several techniques due to
Nekov\'a\v{r} concerning the parity conjecture for self-dual motives.
In particular we show that, for a $p$-adic analytic family, with
irreducible base, of symplectic self-dual global Galois
representations whose $(\vphi,\Ga)$-modules at places lying over $p$
satisfy a Panchishkin condition, the validity of the parity conjecture
is constant among all specializations that are pure.  As an
application, we extend some other results of Nekov\'a\v{r} for Hilbert
modular forms from the ordinary case to the finite-slope case.
\end{abstract}
\tableofcontents

\section{Introduction}

In a series of papers \cite{N2,N3,NDur,AN,NGr,NCan}, Nekov\'a\v{r}
introduced a number of useful techniques for treating the parity
conjecture for self-dual motives, especially for Hilbert modular
forms, at ordinary primes.  This conjecture, which is a piece of the
Bloch--Kato conjecture, compares the global $\vep$-factor with the
dimension modulo $2$ of the Selmer group.  The purpose of the present
paper is to show how recent advances in $p$-adic Hodge theory, namely
papers \cite{P,KPX} by the present authors and K.S.~Kedlaya, can be
used to extend these techniques to finite-slope cases.

In \cite{N3} in particular, it was shown that the validity of the
parity conjecture is constant along certain one-dimensional formal
families that are ordinary in the sense of Panchishkin.  Our main
result extends this conclusion to allow a general rigid analytic space
as the base of the family, and weakens the condition at $p$ to require
that the motive is only Panchishkin on the level of
$(\vphi,\Ga)$-modules.  See Notation~\ref{N:Panchishkin condition} for
the defintion of the latter condition; in the case of modular forms,
it is tantamount to being finite-slope (up to twist), in comparison to
ordinary (up to twist).  We state our main result more precisely as
follows (cf.\ Theorem~\ref{T:main}).  Fix a prime $p$, a number field
$F$, and a finite set of places $S$ of $F$ containing the sets $S_p$
and $S_\infty$ of those lying over $p$ and $\infty$ (respectively).
We write $G_{F,S}$ for the absolute Galois group unramified outside
$S$, and for any prime $v \in S \bs S_\infty$ we write $G_{F_v}$ for a
decomposition group at $v$ with inclusion $G_{F_v} \hookrightarrow
G_{F,S}$.

\begin{theoremstar}
\label{T:thm A}
Let $X$ be an irreducible reduced rigid analytic space over $\Qp$, and
$\scrT$ a locally free coherent $\calO_X$-module equipped with a
continuous, $\calO_X$-linear action of $G_{F,S}$, and a skew-symmetric
isomorphism $j \cn \scrT \stackrel\sim\to \scrT^*(1)$.  Assume given,
for each $v \in S_p$, a short exact sequence
\[
\scrS_v \cn 0 \to \scrD_v^+ \to \bbD_\rig(\scrT|_{G_{F_v}}) \to
\scrD_v^- \to 0
\]
of $(\vphi,\Ga)$-modules over $\calR_{\calO_X}(\pi_{F_v})$, with
$\scrD_v^+$ Lagrangian with respect to $j$.

For a (closed) point $P \in X$, we put $V = \scrT \otimes_{\calO_X}
\kappa(P)$ and $S_v = \scrS_v \otimes_{\calO_X} \kappa(P)$.  Let
$X_\alg$ be the set of points $P \in X$ such that (1) for all $v \in
S_p$, the sequence $S_v$ makes $V|_{G_{F_v}}$ Panchishkin (on the
level of $(\vphi,\Ga)$-modules), and (2) for all $v \in S \bs
S_\infty$ where $\scrT|_{G_{F_v}}$ is ramified, the associated
Weil--Deligne representaiton $WD(V|_{G_{F_v}})$ is pure.

Then the validity of the parity conjecture for $V$, relating the sign
of its global $\vep$-factor to the parity of the dimension of its
Bloch--Kato Selmer group, is independent of $P \in X_\alg$.
\end{theoremstar}

The reader should compare the above statement to \cite[(5.3.1)]{N3}.
We remark that the objects in the statement are well-defined even
without the knowledge of conjectures about motivic $L$-functions, such
as their analytic continuation and functional equation.

Using this result, some techniques of Nekov\'a\v{r} for Hilbert
modular forms of parallel weight two generalize readily from the
ordinary case to the finite-slope case.  The following result gives an
example (cf.\ Theorems~\ref{T:indefinite} and \ref{T:definite}).
(Using a \emph{different} technique, Nekov\'a\v{r} has also proven
this result with no condition on the slope; see \cite{NANT}.)

\begin{theoremstar}
\label{T:thm B}
Let $p \neq 2$ and $F$ a totally real number field. Let $f$ be a
Hilbert modular form for $F$ of parallel weight two with trivial
central character and finite-slope at $p$.  If $[F:\QQ]$ is even and
$f$ is principal series at all finite places, then additionally assume
the condition (A2) of \S\ref{S:Hilbert} and that $\vep(f)=+1$.  Then
the parity conjecture holds for $f$.
\end{theoremstar}

The application to forms of higher weight is genuinely new.  Namely,
the following result follows immediately from combining the two
theorems above, using the triangulation results of \cite{KPX} to
obtain the necessary families $\scrS_v$ of Panchishkin exact sequences
over the relevant eigenvarieties.

\begin{theoremstar}
\label{T:thm C}
Let $p \neq 2$, $F$ a totally real number field with $[F:\QQ]$ odd,
and $X$ an irreducible component of the eigenvariety of finite-slope
overcovergent $p$-adic Hilbert modular cusp forms for $F$.  If $X$
contains a classical, parallel weight two point, then the parity
conjecture holds for all classical points of $X$.
\end{theoremstar}

At least when $[F:\QQ]$ is odd, the preceding result should apply to
every irreducible component $X$ of the eigenvariety, according to the
following folklore conjecture on the \emph{global} geometry of
eigenvarieties, which would generalize results of Buzzard--Kilford
\cite{BK} and Roe \cite{R} when $F=\QQ$ and $p=2$ and $3$,
respectively, and the recent preprint of Liu-Wan-Xiao \cite{LWX} for definite quaternion algebras over $\QQ$ with $p \geq 3$.

\begin{conjstar}
Let $F$ be a totally real number field, and $X$ an irreducible
component of the eigenvariety of finte-slope overconvergent $p$-adic
Hilbert modular cusp forms for $F$.  For every $\epsilon>0$ there
exists an affinoid subdomain $U_\epsilon$ of weight space $W$ such
that the restriction $X \times_W (W \bs U_\epsilon)$ has an
irreducible component $C$ with all slopes strictly less than
$\epsilon$.
\end{conjstar}

In particular, for $\epsilon$ sufficiently small depending on $F$, any
parallel weight two point of slope strictly less than $\epsilon$ is
classical, at least when $X$ is the image under the Jacquet--Langlands
correspondence of an irreducible component of the eigenvariety for a
definite quaternion algebra in \cite{Bu}.  Thus the subspace $C$ of
the conjecture forces $X$ to have classical, parallel weight two
points.

Here is a roadmap of this paper.  In \S\ref{S:review} we review the
$p$-adic Hodge theory and Galois cohomology of $(\vphi,\Ga)$-modules.
All of the results therein are known to the experts, but appear in no
single place in the literature.  In \S\ref{S:parity} we prove
Theorem~\ref{T:thm A}.  While our strategy roughly follows that of
Nekov\'a\v{r}, our generality forces us to replace certain ingredients
of his proof with more technical arguments.  In \S\ref{S:Hilbert} we
explain the applications to Hilbert modular forms, and in particular
prove Theorem~\ref{T:thm B}.

\subsection*{Acknowledgements}

Part of this work was undertaken while J.P. was supported
by an NSF Postdoctoral Research Fellowship.  
L.X. is  partially supported by a grant from the Simons Foundation \#278433 and CORCL research grant from University of California, Irvine.
The authors would like to
thank Ben Howard and Jan Nekov\'a\v{r} for helpful conversations.

\section{Review of $(\vphi,\Ga)$-modules}\label{S:review}

In this section, we recall the facts we will need about
$(\vphi,\Ga)$-modules over the Robba ring.  The first subsection
treats definitions of the objects themselves, the classification of
rank one objects, and the relationship to Galois representations.  In
the next subsection we collect in one place the fundamental invariants
coming from Galois cohomology and $p$-adic Hodge theory, and the
interrelations among them.  The Panchishkin condition is introduced in
the final subsection.

\subsection{Definitions and generalities}

Let $p$ be a prime number, and let $K/\Qp$ be a finite extension with
ring of integers $\calO_K$ and residue field $k$ of cardinality
$q=p^h$.  Fix an algebraic closure $K^\alg$ of $K$ and set $G =
\Gal(K^\alg/K)$ with inertia subgroup $I \subset G$.  Write $K_\infty
= K(\mu_{p^\infty})$, $H = \Gal(K^\alg/K_\infty)$, and $\Ga =
\Gal(K_\infty/K)$; the $p$-adic cyclotomic character $\chi$ on $G$ has
kernel $H$ and maps $G/H = \Ga$ isomorphically onto an open subgroup
of $\Zp^\times$.  (If the dependence on $K$ is relevant, we use a
subscript, e.g. $\Ga_K$.)  Let $k'$ be the residue field of
$K_\infty$, which is finite, and set $K_0 = W(k)[1/p]$ and $K_0' =
W(k')[1/p]$.  Write $e'_K = [K_\infty:\QQ_{p,\infty}]/[k':k]$ for the
ramification degree of $K_\infty/\QQ_{p,\infty}$.  Let
$K_0^\mathrm{ur} = \Qp^\mathrm{ur}$ be the maximal unramified
extension of $K_0$.

Let $\pi_K$ be an indeterminate, and for $r > 0$ put
\[
\calR^r
= \calR^r(\pi_K)
= \{f(\pi_K)=\sum_{n \in \ZZ} a_n\pi_K^n \mid
    a_n \in K_0',
    \text{ convergent for } p^{-r/(p-1)} \leq |\pi_K| < 1\}.
\]
Let $\calR = \calR(\pi_K) = \bigcup_{r \to 0^+} \calR^r(\pi_K)$ be the
\emph{Robba ring}.  The theory of the field of norms takes as input a bit of
choice involving the meaning of $\pi_K$, and outputs the following
variety of structures.  One has a constant $C(\pi_K)>0$ and, for
rational numbers $r \in (0,C(\pi_K))$ (which we henceforth tacitly
assume when we refer to $r$), a finite flat degree $p$ algebra map
$\vphi \cn \calR^r \to \calR^{r/p}$ and a continuous ring-theoretic
action of $\Ga$ on $\calR^r$ that commute with one another, compatible
with the inclusions for varying $r$.  The actions of $\vphi$ and $\Ga$
on the coefficents $K_0'$ are the arithmetic Frobenius action and the natural one, respectively.  For a finite
extension $K'/K$ inside $K^\alg$ and $r <
\min\{C(\pi_K),C(\pi_{K'})e'_{K'}/e'_K\}$, one has canonical
inclusions $\calR^r(\pi_K) \hookrightarrow
\calR^{re'_K/e'_{K'}}(\pi_{K'})$ that are finite \'etale of degree
$[K'_\infty:K_\infty]$ and compatible with the actions of $\vphi$ and
$\Ga_{K'} \subseteq \Ga_K$.

In the case $K=\Qp$, each choice of a basis of $\Zp(1)$ gives rise to
a distinguished choice for the meaning of $\pi_{\Qp}$, the result of
which we call $\pi$.  It satisfies $\vphi(\pi) = (1+\pi)^p-1$ and
$\ga(\pi) = (1+\pi)^{\chi(\ga)}-1$.  Write $t = \log(1+\pi) \in
\calR^r(\pi) \subseteq \calR^{r/e'_K}(\pi_K)$ for Fontaine's element.
One has a family of $\Ga$-equivariant homomorphisms $\iota_n =
\iota_n^{r} \cn \calR^r \hookrightarrow K_\infty[\![t]\!]$, indexed by positive integers
$n$ for which $1/p^{n-1} \leq re'_K$, such that
$\iota_{n+1}^{r/p} \circ \varphi = \iota_n^{r}$.

Fix another finite extension $L/\Qp$ as the coefficient field.  For $?=\emptyset,r$, write
$\calR_L^? = \calR_L^?(\pi_K)$ for the base extension $\calR^?(\pi_K)
\otimes_{\Qp} L$, and extend the actions of $\vphi$ and $\Ga$ to
$\calR_L$ by $L$-linearity.  Note that $\calR_L$ is a product of
B\'ezout domains, so a finite projective $\calR_L$-module of constant
rank is free.

A \emph{$\vphi$-module} is by definition a finite free
$\calR_L$-module $D$ equipped with a semilinear action of $\vphi$ such
that $\vphi(D) \otimes_{\vphi(\calR_L)} \calR_L \stackrel\sim\to D$.
We write $d = d(D) = \rank_{\calR_L} D$ (not to be confused with the
integer $d_L(D)$ appearing in Section~\ref{S:parity} below).  Note
that there always exists $r$ such that $D$ admits a model $D^r$ over
$\calR_L^r$, i.e. a finite free $\calR_L^r$-submodule $D^r \subset D$
such that, writing $D^s = D^r \otimes_{\calR^r} \calR^s \subseteq D$
for $s \in (0,r)$ or $s=\emptyset$, then the inclusion of $\vphi(D^r)$
into $D$ induces an isomorphism $\varphi(D^r)
\otimes_{\varphi(\calR_L^r)} \calR_L^{r/p} \cong D^{r/p}$.  For $s$
sufficiently small, such $D^s$ is uniquely determined.  A
\emph{$(\vphi,\Ga)$-module} is by definition a $\vphi$-module $D$
equipped with a continuous action of $\Ga$ commuting with $\vphi$.
Since $D^s$ is uniquely determined for small $s$, it is stable by
$\Ga$.  If $K'/K$ is a finite extension, we write $D|_{K'} = D
\otimes_{\calR_L(\pi_K)} \calR_L(\pi_{K'})$ for the resulting
$(\vphi,\Ga_{K'})$-module over $\calR_L(\pi_{K'})$.  Conversely, if
$D'$ is a $(\varphi,\Ga_{K'})$-module over $\calR_L(\pi_{K'})$, then
we define $\Ind_{K'}^K D' = \Ind_{\Ga_{K'}}^{\Ga_K} D'$, considered as
a $(\vphi,\Ga_K)$-module over $\calR_L(\pi_K)$.

To each continuous character $\delta \cn K^\times \to L^\times$, there is
associated a rank one $(\varphi, \Gamma)$-module $\calR_L(\delta)$.  When $K=\Qp$, one has
$\calR_L(\delta) = \calR_L e_\delta$, where $\vphi(e_\delta) =
\delta(p)e_\delta$ and $\ga(e_\delta) = \delta(\chi(\ga))e_\delta$ for
$\ga \in \Ga$; for general $K$ we refer to \cite[Construction~6.2.4]{KPX} for the
construction.  Every rank one object $D$ isomorphic to
$\calR_L(\delta_D)$ for a uniquely determined character $\delta_D$.
(In this generality, the fact was originally shown in \cite[Theorem~1.45]{Nak2}
using another language; see \cite[Lemma~6.2.13]{KPX} for a proof in this language.)
For an object $D$ of any rank, we write $\det(D)$ for the character of
$K^\times$ corresponding to the top exterior power of $D$.

There is a slope theory for $\vphi$-modules \cite{Ked}.  Namely, any
$\vphi$-module $D$ is uniquely filtered by a finite list of saturated
subobjects such that each graded piece is pure of some $\vphi$-slope
$\lambda \in \QQ$, and the $\lambda$ are increasing.  Although the
definition of pureness is somewhat technical, morally it means that
$\vphi$ has eigenvalues over a sufficiently large over-ring of
$\calR_L$ of valuation $\lambda$.  We call $D$ \emph{\'etale} if is pure of
slope $0$.  When $D$ is a $(\vphi,\Ga)$-module, the slope filtration
of its underlying $\vphi$-module is, by uniqueness, stable under
$\Ga$; we say $D$ is \emph{\'etale} if its underlying $\vphi$-module is.
Combining work of Fontaine \cite{Fon}, Cherbonnier--Colmez \cite{CC},
and Kedlaya \cite{Ked}, one obtains a fully faithful, exact
$\otimes$-functor $\bbD_\rig$ from the category of continuous
representations of $G$ over $L$ to the category of
$(\vphi,\Ga)$-modules over $\calR_L$, having for essential image the
\'etale objects, and compatible with Galois cohomology and $p$-adic
Hodge theoretic invariants to be introduced in the next subsection.
In the case of rank one, the functor is simple to describe: a rank one
$p$-adic Galois representation is just a character $\delta \cn G \to
L^\times$, and $\bbD_\rig(\delta) = \calR_L(\delta \circ
\mathrm{Art}_K)$, where $\mathrm{Art}_K \cn K^\times \to
G^\mathrm{ab}$ is the local Artin map (normalized to take a
uniformizer onto a lift of geometric Frobenius).

We use a superscript $^*$ to denote internal Hom into the unit
object.  For any $(\vphi,\Ga)$-module $D$, we write $D(n) = D
\otimes_\calR \bbD_\rig(\Qp(1))^{\otimes n}$ for $n \in \ZZ$.

\subsection{Numerology of de~Rham $(\vphi,\Ga)$-modules}\label{S:numerology}

Writing $\Delta \subseteq \Ga$ for the $p$-torsion subgroup (which is
trivial if $p \neq 2$) and fixing a choice of $\ga \in \Ga$ mapping to
a topological generator in $\Ga/\Delta$, for any object $X$ with
commuting actions of $\vphi$ and $\Ga$ (resp. an action of $\Ga$), one
defines $H^*_{\vphi,\ga}(X)$ (resp. $H^*_\ga(X)$) to be the cohomology
of the complex concentrated in degrees $[0,2]$ given by
\[
C^\bullet_{\vphi,\ga}(X) =
[X^\Delta \xrightarrow{\vphi-1,\ga-1} X^\Delta \oplus X^\Delta
 \xrightarrow{\ga-1,1-\vphi} X^\Delta]
\]
(resp. in degrees $[0,1]$ given by $C^\bullet_\ga(X) = [X^\Delta
  \xrightarrow{\ga-1} X^\Delta]$).  One has a canonical identification
$H^0_\ga(X) = X^\Ga$, and an identification $H^1_\ga(X) \approx X_\Ga$
determined up to a constant $\Zp^\times$-multiple (which we fix).
When $X=D$ or $D[1/t]$, where $D$ is a $(\vphi,\Ga)$-module, one calls
$H^*_{\vphi,\ga}(X)$ its \emph{Galois cohomology} and sometimes omits
subscripts for brevity, as in $H^*(X)$.  If $?$ is any subscript, we
write $h^i_?(X) = \dim_L H^i_?(X)$.  Tate's duality and
Euler--Poincar\'e formulas, proven in this context by R.~Liu \cite{L},
show that the $h^*(D)$ are finite, $H^i(D)^* \cong H^{2-i}(D^*(1))$,
and $\sum_i (-1)^i h^i(D) = -[K:\Qp]d$, where $d = \rank_{\calR_L} D$ as before.

Because $H^0(D[1/t])$ is the direct limit of the $H^0(t^{-r}D)$, with
injective transition maps and each term of $L$-dimension at most
$\rank_{\calR_L} D$, we have that $H^0(D[1/t]) = H^0(t^{-r}D)$ for $r
\gg 0$.  Similarly, it is easy to see that $H^0(t^rD^*(1)) = 0$ for
all $r \gg 0$, from which Tate duality gives $H^2(t^{-r}D) = 0$, and
passing to the direct limit gives $H^2(D[1/t]) = 0$.  Furthermore, one
can show that $h^i(t^mD/t^{m+1}D)$ is only nonzero for finitely many
pairs $(i,m)$, from which it follows that the maps $H^1(t^{-r}D) \to
H^1(t^{-(r+1)}D)$ are isomorphisms for $r \gg 0$.  We conclude that
for all $r \gg 0$ and $i=0,1,2$ we have $H^i(D[1/t]) = H^i(t^{-r}D)$,
and therefore $\sum_i (-1)^ih^i(D[1/t]) = \sum_i (-1)^ih^i(t^{-r}D) =
-[K:\Qp]\rank_{\calR_L} D = -[K:\Qp]d$.

If $D$ is a $(\vphi,\Ga)$-module, then one associates to it $p$-adic
Hodge theoretic invariants by the formulas
\[
\bbD_\crys(D) = D[1/t]^\Ga,
\qquad
\bbD_\st(D) = D[1/t,\log\pi]^\Ga,
\qquad
\bbD_\pst(D) = \varinjlim_{K'/K} \bbD_\st(D|_{K'}).
\]
Using the $\vphi$-structure, one can show that the $\Ga$-modules
$\bbD_\dif^+(D) = D^s \otimes_{\calR^s,\iota_n} K_\infty[\![t]\!]$ are
independent of sufficiently small $s$ and large $n$, and one puts
\[
\bbD_\dR^+(D) = \bbD_\dif^+(D)^\Ga,
\qquad
\bbD_\dR(D) = \bbD_\dif^+(D)[1/t]^\Ga.
\]
All these modules are equipped with the usual $\vphi,N,G$, and
filtration structures, and are related in the usual ways as in
Fontaine's theory, the only difference being that the weak
admissibility condition does not appear; see \cite[\S3.1]{N3} for more
information.  Note that, by definition, $H^0(D) = \Fil^0
\bbD_\crys(D)^{\vphi=1}$ and $H^0(D[1/t]) =
\bbD_\crys(D)^{\varphi=1}$.

We define $D$ to be \emph{crystalline} (resp.\ \emph{semistable},
\emph{potentially semistable}, \emph{de~Rham}) if $\dim_{K_0}
\bbD_\crys(D)$ (resp. $\dim_{K_0} \bbD_\st(D)$,
$\dim_{K_0^\mathrm{ur}} \bbD_\pst(D)$, $\dim_K \bbD_\dR(D)$) agrees
with $\rank_\calR D=[L:\QQ_p] d$.  We define \emph{the multiplicity of
  the Hodge--Tate weight $i$} to be $d^i = d^i(D) = \dim_L \Gr^i
\bbD_\dR(D)$ (disregarding the $K$-module structure), and in
particular the cyclotomic character $\chi$ has weight $-1$ with
multiplicity $[K:\Qp]$.  When $L$ contains the images of all
embeddings of $K$ into $\overline \QQ_p$, for each embedding $\sigma:
K \to L$, we define the \emph{multiplicity of the $\sigma$-Hodge--Tate
  weight $i$} to be $d^i_\sigma = d^i_\sigma(D) = \dim_L \left(\Gr^i
\bbD_\dR(D) \otimes_{L \otimes K, 1 \otimes \sigma} L\right)$.

The $p$-adic monodromy theorem asserts that being de~Rham is
equivalent to being potentially semistable.  The analogue of ``weakly
admissible implies admissible'' is that $\bbD_\crys$
(resp.\ $\bbD_\st$, $\bbD_\pst$) induces an equivalence of categories
from crystalline (resp.\ semistable, potentially semistable)
$(\vphi,\Ga)$-modules and filtered $\vphi$-modules (resp.\ filtered
$(\vphi,N)$-modules, filtered $(\vphi,N,G)$-modules) in a manner
taking the $\vphi$-Newton polygon (\`a la Kedlaya) on the source to
the Newton-minus-Hodge polygon on the target.

\begin{lemma}\label{L:HT90}
Let $0 \to D_1 \to D_2 \to D_3 \to 0$ be a short exact sequence of
de~Rham $(\vphi,\Ga)$-modules.  Then for $?=\st,\pst$ the sequence $0
\to \bbD_?(D_1) \to \bbD_?(D_2) \to \bbD_?(D_3) \to 0$ is exact.
\end{lemma}

\begin{proof}
The claim for $?=\pst$ follows from the $p$-adic monodromy theorem,
and one deduces from this the claim for $?=\st$ because Hilbert's
Theorem~90 shows that taking $G$-invariants is exact in this case.
\end{proof}

We will have need for the following $\Ga$-equivariant version of
elementary divisors over $K_\infty[\![t]\!]$, whose proof we leave as an exercise.

\begin{lemma}\label{L:elementary divisors}
Suppose $V$ is a (nonzero) finite-dimensional $K$-vector space,
considered with trivial $\Ga$-action, and let $\calL = V \otimes_K
K_\infty[\![t]\!]$ and $\calW = \calL[1/t]$ have the obvious
semilinear actions of $\Ga$.  Let $\calL'$ be another $\Ga$-stable
$K_\infty[\![t]\!]$-lattice in $\calW$.  Then there exist a $K$-basis
$\{v_i\}$ of $V$ (hence also a $K_\infty[\![t]\!]$-basis of $\calL$)
and integers $a_i$ such that $\calL'$ is the $K_\infty[\![t]\!]$-span
of the $t^{a_i} v_i$.
\end{lemma}

\emph{For the rest of this section, assume that $D$ is de~Rham}, so that by definition
\[
h^0_\ga(\bbD_\dif^+(D)[1/t]) = [K:\Qp]d = d_- + d_+,
\]
where $d_- = d_-(D)$ (resp. $d_+ = d_+(D)$) are
the number of negative (resp. nonnegative) Hodge--Tate weights of $D$. (The multiplicity $d_-$, which is only used in this Section~\ref{S:review}, should not be confused with the integer $d^-$ appearing in Section~\ref{S:parity}.)
The $\Ga$-equivariant elementary divisor theory, applied to $V =
\bbD_\dR(D)$ and $\calL' = \bbD_\dif^+(D)$ (forgetting the structure of the coefficient field $L$), shows that
$\bbD_\dif^+(D)$ is isomorphic, as a $K_\infty\llbracket t \rrbracket$-module with $\Gamma$-action, to
$\bigoplus_{i=1}^{[L:\QQ_p]d} t^{a_i}K_\infty[\![t]\!]$ for integers $a_i$.  
Each factor $t^{a_i}K_\infty\llbracket t\rrbracket$ for $a_i$ positive (resp. $a_i$ nonpositive) contributes $1/[L:\Qp]$ to the multiplicity of the negative (resp. nonnegative) Hodge--Tate weight $-a_i$ on the one hand, and contributes $0$ (resp. $1/[L:\Qp]$) to $h_\gamma^i(\bbD_\dif^+(D))$ for $i=0,1$ on the other hand.
From this, one
deduces easily that, for
$i=0,1$, one has $h^i_\ga(\bbD_\dif^+(D)[1/t]) = [K:\Qp]d$, $h^i_\ga(\bbD_\dif^+(D)) = d_+$, and the natural
map $H^i_\ga(\bbD_\dif^+(D)) \to H^i_\ga(\bbD_\dif^+(D)[1/t])$ is
injective.  Connecting homomorphisms $\delta^i_D$ for $i=0,1$, are
constructed in \cite[Theorem~2.8]{Nak} so that the sequence
\begin{align}
\label{E:fundamental exact sequence}
0 \longrightarrow& H_{\vphi,\ga}^0(D)
 \longrightarrow H_{\vphi,\ga}^0(D[1/t]) \oplus H^0_\ga(\bbD_\dif^+(D))
 \xrightarrow{\ -\ } H^0_\ga(\bbD_\dif^+(D)[1/t]) \\
 \nonumber
\xrightarrow{\; \delta^0_D\; } & H_{\vphi,\ga}^1(D)
 \longrightarrow H_{\vphi,\ga}^1(D[1/t]) \oplus H^1_\ga(\bbD_\dif^+(D))
 \xrightarrow{\ -\ } H^1_\ga(\bbD_\dif^+(D)[1/t])\\
 \nonumber
 \xrightarrow{\;\delta^1_D\;} & H_{\vphi,\ga}^2(D) \longrightarrow 0
\end{align}
is exact, where ``$-$'' denotes the difference of the natural maps.
One should view \eqref{E:fundamental exact sequence} as a
generalization of the fundamental exact sequence in $p$-adic Hodge
theory.

One defines subspaces $H^1_?(D)$ of $H^1(D)$, for $?=e,f,\st,g+,g$, by
the formulas
\begin{align*}
H_e^1(D) &= \ker\big( H^1_{\vphi,\ga}(D) \to H^1_{\vphi,\ga}(D[1/t]) \,\big), \\
H_f^1(D) &= \ker\big( H^1_{\vphi,\ga}(D) \to H^1_\ga(D[1/t])\,\big), \\
H_\st^1(D) &= \ker\big( H^1_{\vphi,\ga}(D) \to H^1_\ga(D[1/t,\log\pi])\,\big), \\
H_{g+}^1(D) &= \ker\big( H^1_{\vphi,\ga}(D) \to H^1_\ga(\bbD_\dif^+(D))\,\big), \\
H_g^1(D) &= \ker\big( H^1_{\vphi,\ga}(D) \to H^1_\ga(\bbD_\dif^+(D)[1/t])\,\big).
\end{align*}
One may rephrase these definitions as follows.  Define $K_e = \Qp$,
$K_f = K_\st = K_0$, and $K_{g+} = K_g = K$, and define $\bbD_e(D) =
H^0(D[1/t])$, $\bbD_f = \bbD_\crys$, $\bbD_\st$ as given above,
$\bbD_{g+} = \bbD_\dR^+$, and $\bbD_g = \bbD_\dR$, so that $\bbD_?(D)$
is a $(K_? \otimes_{\QQ_p} L)$-module for each of $?=e,f,\st,g+,g$.  If the class $c \in
H^1(D)$ corresponds to the extension of $(\vphi,\Ga)$-modules $0 \to D
\to E_c \to \calR \to 0$, then one has $c \in H^1_?(D)$ if and only if
$\rank_{K_? \otimes_{\Qp} L} \bbD_?(E_c) = \rank_{K_? \otimes_{\Qp} L}
\bbD_?(D)+1$ (the ranks considered as locally constant functions on
$\Spec(K_? \otimes_{\Qp} L)$).

There are obvious inclusions $H^1_e(D) \subseteq H^1_f(D) \subseteq
H^1_\st(D) \subseteq H^1_g(D)$ and $H^1_{g+}(D) \subseteq H^1_g(D)$.
Two of these inclusions are in fact equalities: applying Lemma~\ref{L:HT90}
to the reformulated definitions shows that $H^1_\st(D) = H^1_g(D)$, and
because $H^i_\ga(\bbD_\dif^+(D)) \hookrightarrow
H^i_\ga(\bbD_\dif^+(D)[1/t])$ one has $H^1_{g+}(D) = H^1_g(D)$.  The
injectivity of $H^1_\ga(\bbD_\dif^+(D)) \to
H^1_\ga(\bbD_\dif^+(D)[1/t])$ implies that $\img \delta^0_D =
H^1_e(D)$, whence $h^1_e(D) = d_- + h^0(D) - h^0(D[1/t])$, and it is
shown in \cite[Corollary~1.4.5]{Ben} that $h^1_f(D) = d_- + h^0(D)$.
It is also shown in \cite[Corollary~1.4.10]{Ben} that $H^1_f(D)$ and
$H^1_f(D^*(1))$ are orthogonal complements under Tate duality, and
similarly the commutative diagram of perfect pairings in \cite[Lemma~2.15(1)]{Nak}, namely
\[
\xymatrix{
H^0_\gamma(\bbD^+_\dif(D)[1/t]) \ar[d]^{\delta_D^0} \ar@{}[r]|{\times} & H^1_\gamma(\bbD^+_\dif(D^*(1))[1/t]) \ar[r] & K \otimes_{\Qp} L \ar[d]^{\mathrm{Tr}_{K/\QQ_p \otimes \mathrm{id}}} 
\\
H^1_{\vphi, \gamma}(D) \ar@{}[r]|\times & H^1_{\vphi, \gamma}(D^*(1)) \ar[u]_\iota \ar[r] & L,
}
\]
implies that $\img \delta_D^0 = H^1_e(D)$ is the  orthogonal complement of $\ker(\iota) = H^1_g(D^*(1))$.  It
follows that $h^1_{/f}(D) = h^1_f(D^*(1))$ and $h^1_{/g}(D) =
h^1_e(D^*(1))$, where we write $h^1_{/?}(D) = h^1(D)-h^1_?(D) = \dim_L
H^1(D)/H^1_?(D)$ for $?=e,f,g$.  We summarize the dimensions as
follows.

\begin{formulary}\label{F:formulary}
\begin{align*}
h^0(D) = h^2(D^*(1))
 &= \dim_L \Fil^0 \bbD_\crys(D)^{\vphi=1}, \\
h^1(D) = h^1(D^*(1))
 &= [K:\Qp]d + h^0(D) + h^0(D^*(1)), \\
h^0(D[1/t])
 &= \dim_L \bbD_\crys(D)^{\vphi=1}, \\
h^1(D[1/t]) &= [K:\Qp]d + h^0(D[1/t]), \\
h^2(D[1/t]) &= 0, \\
h^1_e(D) = h^1_{/g}(D^*(1))
 &= d_- + h^0(D) - h^0(D[1/t]), \\
h^1_f(D) = h^1_{/f}(D^*(1))
 &= d_- + h^0(D), \\
h^1_g(D) = h^1_{g+}(D) = h^1_\st(D)
 &= d_- + h^0(D) + h^0(D^*(1)[1/t]).
\end{align*}
\end{formulary}

For completeness, we describe the graded pieces of the filtration
\[
0 \subseteq H^1_e(D) \subseteq H^1_f(D) \subseteq H^1_g(D) \subseteq
H^1(D).
\]
From the exact sequence \eqref{E:fundamental exact sequence} one sees immediately that
\[
H^1_e(D)
= \img \delta^0_D \cong
\bbD_\dR(D)/(\bbD_\dR^+(D)+\bbD_\crys(D)^{\vphi=1}).
\]
Hence, also, $H^1(D)/H^1_g(D)$ is dual to
$\bbD_\dR(D^*(1))/(\bbD_\dR^+(D^*(1))+\bbD_\crys(D^*(1))^{\vphi=1})$.
Similarly, it follows from the short exact sequence
\[
0 \to H^0_\ga(D[1/t])/(\vphi-1) \to H^1_{\vphi, \ga}(D[1/t]) \to
H^1_\ga(D[1/t])^{\vphi=1} \to 0
\]
that $H^1_f(D)/H^1_e(D) \hookrightarrow H^0_\ga(D[1/t])/(\vphi-1) =
\bbD_\crys(D)/(\vphi-1)$, and comparing the dimension on the right
with Formulary~\ref{F:formulary} we see that this inclusion is an
isomorphism.  Hence, also, $H^1_g(D)/H^1_f(D)$ is dual to
$\bbD_\crys(D^*(1))/(\vphi-1)$.

\begin{remark}
All numerical invariants of $D$ defined in this section have been
arranged to be invariant under replacing $L$ by a finite extension
$L'$, and $D$ by the $(\vphi,\Ga)$-module over $\calR_{L'}$ given by
$D \otimes_L L' = D \otimes_{\calR_L} \calR_{L'}$.  Therefore, in
arguments involving these constants we may enlarge $L$ in this manner
without any harm.
\end{remark}

\subsection{The Panchishkin case}

\begin{notation}
\label{N:Panchishkin condition}
We say a $(\varphi,\Gamma)$-module $D$ is \emph{Panchishkin} if it sits in a
short exact sequence $S: 0 \to D^+ \to D \to D^- \to 0$ of
$(\vphi,\Ga)$-modules with each $D^\pm$ de~Rham with $d_\mp(D^\pm) = [K:\Qp]d(D^\pm)$, i.e. $D^+$ (resp. $D^-$) only has negative (resp. nonnegative) Hodge--Tate weights.  
\end{notation}

We assume that $D$ is Panchishkin for the rest of this section. We
warn the reader that unlike in the case of Galois representations, the
sequence $S$ is not necessarily uniquely determined by $D$. Using
$\Ga$-equivariant elementary divisors, Lemma~\ref{L:elementary
  divisors}, the Panchishkin condition implies that $D$ is de~Rham,
and $d_\pm = d_\pm(D) = [K:\Qp]d(D^\mp)$.  (In fact $D$ is semistable
if both $D^\pm$ are, by Lemma~\ref{L:HT90}.)  One also deduces from
$\Ga$-equivariant elementary divisors that the extension of
$K_\infty[\![t]\!]$-modules with $\Gamma$-actions
\[
0 \to \bbD_\dif^+(D^+) \to \bbD_\dif^+(D) \to \bbD_\dif^+(D^-) \to 0
\]
is split, that for $i=0,1$ one has $H^i_\ga(\bbD_\dif^+(D^+))=0$ so
$H^i_\ga(\bbD_\dif^+(D)) \stackrel\sim\to H^i_\ga(\bbD_\dif^+(D^-))$,
and that $h^i_\ga(\bbD_\dif^+(D^-)) = d_+$.  In particular,
$H^1_g(D^+) = H^1(D^+)$.  Letting $\delta^i_S \cn H^i(D^-) \to
H^{i+1}(D^+)$ denote the usual connecting map for the long exact
cohomology sequence for $S$, a simple diagram chase shows that the
sequence
\begin{equation}
\label{E:long exact sequence with g}
0 \to H^0(D^+) \to H^0(D) \to H^0(D^-)
\xrightarrow{\delta^0_S} H^1_g(D^+) \to H^1_g(D) \to H^1_g(D^-)
\to X \to 0,
\end{equation}
is exact, where $X$ satisfies
\[
\dim_L X = h^0((D^+)^*(1)[1/t]) + h^0((D^-)^*(1)[1/t]) -
h^0(D^*(1)[1/t]).
\]

\begin{lemma}\label{L:panchishkin H1f}
Let $D$ be Panchishkin, and assume that $\bbD_\crys(D)^{\vphi=1} =
\bbD_\crys(D^*(1))^{\vphi=1} = 0$.  Then $\delta_S^0$ and the
functorial map $H^1(D^+) \to H^1(D)$ fit into a short exact sequence
\[
0 \to H^0(D^-) \xrightarrow{\delta^0_S} H^1(D^+) \to H^1_f(D) \to 0.
\]
\end{lemma}

\begin{proof}
In fact, we already have $H^1_g(D^+) = H^1(D^+)$, and
Formulary~\ref{F:formulary} gives $H^0(D) = H^0(D^+) = 0$ and
$H^1_f(D) = H^1_g(D)$, so it suffices to check the surjectivity on the
right. 
In light of \eqref{E:long exact sequence with g}, the desired cokernel has dimension 
$
\dim H^1_g(D^-) - \dim_L X$.
Noting that $d_-(D^-) = 0$ and our hypothesis implies $h^0(D^*(1)[1/t])=0$,  Formulary~\ref{F:formulary}
shows that 
\[
\dim H^1_g(D^-) - \dim_L X = h^0(D^-) -  h^0((D^+)^*(1)[1/t]).
\]
Since this dimension is a nonnegative, it gives the second inequality in
\[
h^0(D^-[1/t]) \geq h^0(D^-) \geq h^0((D^+)^*(1)[1/t]).
\]
Since our hypotheses are symmetric under replacing $D$
(resp.\ $D^\pm$, $S$) with $D^*(1)$ (resp.\ $(D^\mp)^*(1)$, $S^*(1)$),
the reverse inequality also holds, and in fact there is equality
throughout.  The claim follows.
\end{proof}

\section{The parity conjecture in families}\label{S:parity}

In this section we show how to generalize \cite{N3} (always tacitly
taking into account \cite{N3e}) to our setting.
Our exposition will try to be self-contained and at the same time include a complete discussion of the compatibility with that paper.

In the first subsection we treat those purely local facts taking place
at primes over $p$, essentially generalizing \cite[\S3]{N3}.  The
second subsection, which is preparatory, gives definition and basic
facts concerning rational functions on rigid analytic spaces.  In the
third subsection, we treat those purely local facts taking place at
primes not over $p$.  In the final subsection we treat the purely
global facts, essentially generalizing \cite[\S\S4--5]{N3}.

In the first and third subsections, we assume the reader is familiar
with Weil--Deligne representations and local $\vep$-factors as in
\cite[\S\S1--2]{N3}, whose notations and claims we use without
modification.  Namely, $K/\Qell$ denotes a finite extension with
residue field $k$ of cardinality $q = \ell ^h$.  We write $G_K$
(resp.\ $I_K$, $W_K$) for its absolute Galois group (resp.\ inertia
group, Weil group), and $\nu: W_K/I_K \xrightarrow{\sim} \ZZ$ for the
isomorphism which sends a lift $f \in G_K$ of the geometric Frobenius
(i.e. inducing $x \mapsto x^{q^{-1}}$ on $k$) to $1$.  Let
$\unr(\alpha)$ be the one-dimensional Weil--Deligne representation
over $L$ with $N=0$, $I$ acting trivially, and $f$ acting through
multiplication by $\alpha$.  For example, $WD(\Qp(1)) = \unr(q^{-1})$.
One has $\unr(\alpha)^* \cong \unr(\alpha^{-1})$.  Recall the notation
$\spp(m)$ from \cite[Example~1.2.3]{N3}, and that $\spp(m)^* \cong
\spp(m)(1-m)$.  In paticular, $(\unr(\alpha) \otimes \spp(m))^*(1)
\cong \unr(q^{m-2}\alpha^{-1}) \otimes \spp(m)$.  The object
$\unr(\alpha) \otimes \spp(m)$ is pure of weight $n$ if and only if
$\alpha$ is a $q$-Weil number of weight $n+m-1$.

Note that rank one Weil--Deligne representations are identified with
certain characters of the Weil group, so composing with the local
Artin map allows us to view them as characters of $K^\times$, just as
is done with rank one representations of the Galois group.  For
example, $\unr(\alpha)$ is identified with the character that is
trivial on $\calO_K^\times$ and sends a uniformizer to $\alpha$.  For
any Weil--Deligne representation $\Delta$, we let $\det(\Delta)$ be
the character of $K^\times$ corresponding to the top exterior power of
$\Delta$.

\subsection{Local arguments: $\ell = p$}\label{S:local1}

This subsection is concerned with \cite[\S3]{N3}.  Here we let $K$ be
as in the start of this section, and we assume that the characteristic
$\ell$ of $k$ is equal to $p$.  We also fix a finite extension
$L/\Qp$.

The basic invariants of $p$-adic Galois reprsentations of $p$-adic
fields in \cite[\S3.1]{N3} (namely, Galois cohomology, Fontaine's
functors, and Bloch--Kato's subspaces) are replaced by their
generalizations to $(\vphi,\Ga)$-modules in the preceding section, but
no other modifications are necessary to that subsection.

Fix a field extension $E/K_0^\mathrm{ur}$ admitting an embedding $\tau
\cn L \hookrightarrow E$.  Fontaine's recipe for attaching a
Weil--Deligne representation $WD_\tau(V)$ to a semistable
representation representation $V$ of $G_K$ over $L$, described in
\cite[\S3.2]{N3}, passes first to $\bbD_\pst(V)$, and from
$\bbD_\pst(V)$ to $WD_\tau(V)$.  Since the second step does not
require the weak admissibility of $\bbD_\pst(V)$, it may applied
verbatim to $\bbD_\pst(D)$ for $D$ a potentially semistable
$(\vphi,\Ga)$-module to obtain $WD_\tau(D)$.  Namely, if $\rho_D$
denotes the usual $G_K$-action on $\bbD_\pst(D)$, the space
\[
WD_\tau(D)= \bbD_\pst(D) \otimes_{K_0^\mathrm{ur} \otimes_{\Qp} L, \id \otimes \tau} E
\]
with $w \in W_K$ acting $E$-linearly by $\rho_D(w) \circ
\vphi^{h\nu(w)} \otimes \id$ and $N$ acting by $N \otimes \id$ is a
Weil-Deligne representation over $K$.  The isomorphism class of
$WD_\tau(D)$ does not depend on $\tau$, so we abbreviate it to
$WD(D)$.
 
We shall need the following invariants for a potentially semistable
$(\vphi, \Ga)$-module $D$ (which should not be confused with the
$\calR_L$-rank $d(D)$ and the multiplicity $d_-(D)$ appearing in
Section~\ref{S:review}):
\[
d_L(D) = \sum_{i \in \ZZ} i \dim_L \Gr^i\bbD_\dR(D), \quad d^-(D) = \sum_{i <0} i \dim_L \Gr^i\bbD_\dR(D).
\]

The only further modification we require to \cite[\S3.2]{N3} is to
Proposition~3.2.6, which is replaced by the following analog.
\begin{lemma}
\label{L:det = det}
For each potentially semistable $(\vphi,\Gamma)$-module $D$, we have
\begin{equation}
\label{E:det = det}
(\mathrm{det}_E(WD(D)))(-1) = (-1)^{d_L(D)} (\det(D) )(-1).
\end{equation}
\end{lemma}
\begin{proof}
As in the proof of \cite[Proposition~3.2.6]{N3} (with $V$ replaced by $D$), it suffices to treat the case when $D$ is of rank one, and hence is associated to a continuous character $\chi_D: K^\times \to L^\times$.  The potential semistability of $D$ forces $\chi_D$ to take the form of $ \chi \prod_{\sigma: K \hookrightarrow L}
(\sigma \circ \chi_\pi)^{-n_\sigma}$ with $\chi$ a character for which $\chi(\calO_K^\times)$ is finite, and $\chi_\pi$ the Lubin-Tate character defined as in \cite[Example~3.2.2(4)]{N3}.
The equality \eqref{E:det = det} for the character $\sigma\circ \chi_\pi$ is treated in the proof of \cite[Proposition~3.2.6]{N3} and for the character $\chi$ is an easy direct computation.
\end{proof}

Some statements in \cite[\S3.3]{N3} become false for
$(\vphi,\Ga)$-modules.  In Definition~3.3.1 therein (and replaced by Notation~\ref{N:Panchishkin condition}), the $D^{\pm}$ are
no longer uniquely determined.  In Proposition~3.3.2(1), the first
equality holds (for a Panchishkin $D$), but the second equality may now fail (although it is
never used in that paper or this one).  Proposition~3.3.2(2) is
replaced by Lemma~\ref{L:panchishkin H1f} above; note that not
requiring any restriction to $K'/K$ over which $D$ becomes semistable
gives a slightly stronger result.

We conclude this subsection with a discussion of
\cite[Proposition~3.3.3]{N3} (by replacing it with
Proposition~\ref{P:panchishkin epsilon} below).  Since uniqueness of
the Panchishkin filtration fails, in our generality we will have to
\emph{assume} the conclusion of part (1), that $j$ induces
isomorphisms $D^\pm \stackrel\sim\to (D^\mp)^*(1)$.  Part (3) holds
verbatim which we include as \eqref{E:det + / det +} below.  We ignore
part (2) in the present paper, which is only used in \cite{N3} to
prove part (4), and we now prove (4) directly.

\begin{prop}\label{P:panchishkin epsilon}
Let $D$ be a Panchishkin $(\vphi,\Ga)$-module, pure (of weight $-1$),
with a symplectic pairing $j \cn D \stackrel\sim\to D^*(1)$ that
induces isomorphisms $D^\pm \stackrel\sim\to (D^\mp)^*(1)$.  Put
$\Delta^? = WD(D^?)$ for $?=\emptyset,\pm$.  Then
\begin{equation}
\label{E:det + / det +}
(\mathrm{det}_E(\Delta^+))(-1) \big / \det(D^+)(-1) = (-1)^{d_L(D^+)} = (-1)^{d^-(D)}, \textrm{ and}
\end{equation}
\begin{equation}\label{E:panchishkin epsilon}
\vep(\Delta) = (-1)^{h^0(D^-)}(-1)^{d^-(D)}\det(D^+)(-1).
\end{equation}
\end{prop}

\begin{proof}
The equality \eqref{E:det + / det +} follows from applying
Lemma~\ref{L:det = det} to $D^+$, so the rest of this proof is
concerned with \eqref{E:panchishkin epsilon}.  We first discuss the
term $(-1)^{h^0(D^-)}$.  Since all the Hodge--Tate weights of $D^-$
are nonnegative, one has
\begin{equation}\label{E:h^0(D^-)}
h^0(D^-) = \dim_L \bbD_\crys(D^-)^{\vphi=1}
= \dim_E \Delta^{-,I,N=0,f=1}.
\end{equation}
Letting $V^{A\sim\lambda}$ denote the generalized $\lambda$-eigenspace
of a linear operator $A$ on a vector space $V$, we claim that
\[
\dim_E \Delta^{-,I,N=0,f=1}
\equiv
\dim_E \Delta^{-,I,N=0,f\sim1}
\pmod2.
\]
This follows from \cite[Lemma~2.2.2]{N3} as soon as we show that $f$
is orthogonal for a nondegenerate, symmetric pairing on
$\Delta^{-,I,N=0,f\sim1}$.  To exhibit this pairing, consider the
diagram
\[\begin{array}{r@{\ }c@{\ }c@{\ }c@{\ }c@{\ }c@{\ }l}
0 \to & \Delta^{+,f\sim1} & \to & \Delta^{f\sim1} & \to &
  \Delta^{-,f\sim1} & \to 0 \\
& \downarrow & & \downarrow & & \downarrow \\
0 \to & \Delta^{+,f\sim q^{-1}} & \to & \Delta^{f\sim q^{-1}} & \to &
  \Delta^{-,f\sim q^{-1}} & \to 0,
\end{array}\]
where the vertical arrows are $N$.  Our purity hypothesis on $V$ shows
that the middle vertical arrow is an isomorphism, so the snake lemma
gives an isomorphism $a \cn \Delta^{-,N=0,f\sim1} \cong
\Delta^{+,f\sim q^{-1}}/N$.  On the other hand, $j$ induces $b \cn
\Delta^{+,f\sim q^{-1}}/N \cong (\Delta^{-,N=0,f\sim1})^*(1)$.  Write
$\langle\cdot,\cdot\rangle_{b\circ a} \cn \Delta^{-,N=0,f\sim1} \times
\Delta^{-,N=0,f\sim1} \to E(1)$ for the nondegenerate pairing adjoint
to $b \circ a$, and $\langle\cdot,\cdot\rangle_j \cn \Delta \times
\Delta \to E(1)$ for the nondegenerate pairing adjoint to $j$.  Given
$\delta,\delta' \in \Delta^{-,N=0,f\sim1}$, if
$\tilde\delta,\tilde\delta' \in \Delta^{f\sim1}$ denote arbitrary
lifts (respectively), one computes from the definitions of $a$ and $b$
that $\langle\delta,\delta'\rangle_{b\circ a} = \langle
N(\tilde\delta),\tilde\delta'\rangle_j$.  Because $N$ is adjoint to
$-N$ and $j$ is symplectic, it follows that
$\langle\cdot,\cdot\rangle_{b \circ a}$ is symmetric.  To see that $f$
is orthogonal, we note that $f\tilde\delta,f\tilde\delta'$ are lifts
of $f\delta,f\delta'$ (respectively), and compute that
\[
\langle N(f\tilde\delta),f\tilde\delta'\rangle_j
= p\langle fN(\tilde\delta),f\tilde\delta'\rangle_j
= pf\langle N(\tilde\delta),\tilde\delta'\rangle_j
= \langle N(\tilde\delta),\tilde\delta'\rangle_j.
\]
Since $\langle\cdot,\cdot\rangle_{b \circ a}$ is $I$-invariant by
construction, the desired pairing on $\Delta^{-,I,N=0,f\sim1}$ is
obtained by restriction.

We return to considering the desired formula \eqref{E:panchishkin
  epsilon}.  Each term appearing in it is left unchanged upon
replacing $\Delta$ by $\Delta^{f\text{-ss}}$.  For the left hand side,
this is \cite[(2.1.2.1)]{N3}.  For the term $(-1)^{h^0(D^-)}$, this
was just shown.  For the remaining terms, this is obvious.  So, we
assume always that $\Delta$ is Frobenius-semisimple.

Now $\Delta^{N\text{-ss}}$ is semisimple, so it splits into
$\Delta^{+,N\text{-ss}} \oplus \Delta^{-,N\text{-ss}} \cong
\Delta^{+,N\text{-ss}} \oplus (\Delta^{+,N\text{-ss}})^*(1)$, and the
argument of \cite[Proposition~2.2.1(4)]{N3} shows that
$\vep(\Delta^{N\text{-ss}}) = \det(\Delta^{+,N\text{-ss}})(-1)$.  We
compute that the right hand side of \eqref{E:panchishkin epsilon} is
\begin{align*}
&(-1)^{h^0(D^-)} (-1)^{d^-(D)} \det(D^+)(-1) \\
&\qquad= (-1)^{h^0(D^-)} \det(\Delta^+)(-1)
  &&\text{by \eqref{E:det + / det +}} \\
&\qquad= (-1)^{h^0(D^-)} \det(\Delta^{+,N\text{-ss}})(-1)
  &&\text{because rank one objects have $N=0$} \\
&\qquad= (-1)^{h^0(D^-)} \vep(\Delta^{N\text{-ss}})
  &&\text{by \cite[Proposition 2.2.1(4)]{N3}} \\
&\qquad= (-1)^{h^0(D^-)} \vep(\Delta)/\det(-f|\Delta^I/\Delta^{I,N=0})
  &&\text{by \cite[(2.1.2.3)]{N3}}.
\end{align*}
Combining this with \eqref{E:h^0(D^-)}, our desired identity
\eqref{E:panchishkin epsilon} is thus equivalent to the claim that
\begin{equation}\label{E:reduction1}
(-1)^{\dim_E \Delta^{-,I,N=0,f=1}}
= \det(-f|\Delta^I/\Delta^{I,N=0}).
\end{equation}
Recall that $I$ acts semisimply, and moreover any Weil--Deligne
representation $\Delta_0$ is functorially a direct sum $\Delta_0^I
\oplus \Delta_0'$, where $\Delta_0'$ is the component where $I$ acts
nontrivially.  After applying this decomposition to $\Delta_0 =
\Delta^?$, $?=\emptyset,\pm$, in showing \eqref{E:reduction1} we may
assume that inertia acts trivially.  Multiplying both sides of
\eqref{E:reduction1} through by $\det(-f|\Delta^{N=0})$, and using the
evenness of $d = \dim_E \Delta$ and \cite[(1.2.2)(2)]{N3} to compute
the resulting right hand side, we are left to show that
\begin{equation}\label{E:reduction2}
(-1)^{\dim_E\Delta^{-,N=0,f=1}} \det(-f|\Delta^{N=0}) = q^{-d/2}.
\end{equation}

If necessary, enlarge $E$ to include the eigenvalues of $f$ on
$\Delta$.  One can decompose $\Delta$ into a sum of subobjects
$\Delta_i$ of dimension $d_i$, pairwise orthogonal for the symplectic
pairing, each pair $(\Delta_i,j|_{\Delta_i})$ isomorphic to one of
\begin{itemize}
\item[(1)] $\unr(\alpha_i) \otimes \spp(d_i/2) \oplus
  \unr(q^{d_i/2-2}\alpha_i^{-1}) \otimes \spp(d_i/2)$ with $\alpha_i$
  a $q$-Weil number of weight $d_i/2-2$, equipped with the symplectic
  form $\langle (x,x'),(y,y')\rangle = x'(y)-y'(x)$,
\item[(2)] $\unr(\alpha_i) \otimes \spp(d_i)$ with $\alpha_i = \pm_i
  q^{d_i/2-1}$, equipped with its unique symplectic form (a twist of
  the one in \cite[Example~1.2.3]{N3}).
\end{itemize}
This follows from \cite[(1.2.4)]{N3}: Note that $\Delta$ has trivial
inertia action and is Frobenius-semisimple.  In the case (1) we may
assume that the term $X$ appearing in loc.\ cit.\ is of the form
$\unr(\alpha) \otimes \spp(m)$, with the condition on $\alpha_i$
coming from purity.  In the case (2) the term $\rho$ appearing in
loc.\ cit.\ is of the form $\unr(\alpha) \otimes \spp(m)$, with the
fact that $\alpha=\pm q^{m/2-1}$ being forced by the self-dual
condition.

In the case where $\Delta_i$ is of the form (1), one immediately
computes that $\Delta_i^{N=0} \cong \unr(q^{1-d_i/2}\alpha_i) \oplus
\unr(q^{1-d/2}q^{d_i/2-2}\alpha_i^{-1})$, so $\det(-f|\Delta_i^{N=0})
= q^{-d_i/2}$.  In the case where $\Delta_i$ is of the form (2), one
has $\Delta_i^{N=0} \cong \unr(q^{1-d}\alpha)$, so
$\det(-f|\Delta_i^{N=0}) = \mp_i q^{-d_i/2}$.  Thus,
\eqref{E:reduction2} is equivalent to the claim that $\dim_E
\Delta^{-,N=0,f=1}$ has the same parity as the number of factors of
the form (2) with $\pm_i = +$.

Functorially splitting $\Delta^?$, where $?=\emptyset,\pm$, in the
form $\Delta' \oplus \Delta''$, where $f$ has eigenvalues on $\Delta'$
(resp.\ $\Delta''$) belonging (resp.\ not belonging) to $q^\ZZ$, we
replace $\Delta^?$ with its first summand because the second makes no
contribution to this congruence.  Since each $\Delta_i$ of the form
(1) has two indecomposible summands, and each $\Delta_i$ of the form
(2) has one indecomposible summand and now has $\pm_i = +$, the claim
\eqref{E:reduction2} is equivalent to the claim that $\dim_E
\Delta^{-,N=0,f=1}$ has the same parity as the number of
indecomposible summands of $\Delta$.  To count the latter, purity
implies that each indecomposible summand $\Delta_0$ of $\Delta$ as
Weil--Deligne representation has $\Delta_0^{f=1}$ one-dimensional, so
that the number of indecomposible summands is equal to $\dim_E
\Delta^{f=1}$.  Thus we are to show that
\begin{equation}
\label{E:delta f=1 mod 2}
\dim_E \Delta^{-,N=0,f=1}
\equiv
\dim_E \Delta^{f=1}
\pmod2.
\end{equation}
Write $r_\pm = \dim_E \Delta^{\pm,f=1}$ so that the right hand side of
\eqref{E:delta f=1 mod 2} is $r_-+r_+$, and note that duality gives
$r_+ = \dim_E \Delta^{-,f=q^{-1}}$.  Purity implies that $N \cn
\Delta^{f=1} \to \Delta^{f=q^{-1}}$ is an isomorphism, so $N \cn
\Delta^{-,f=1} \to \Delta^{-,f=q^{-1}}$ is surjective.  It follows
that the left hand side of \eqref{E:delta f=1 mod 2} is $r_--r_+$.  As
$r_++r_- \equiv r_--r_+ \pmod2$, this completes the proof of the
proposition.
\end{proof}

\subsection{Rational functions on rigid analytic spaces}

In this preparatory subsection, we let $L$ be a complete
nonarchimedean field, and we let $X$ be a reduced rigid analytic space
over $L$.  We briefly present here the definition and basic properties
of the $L$-algebra $\kappa(X)$ of rational functions on $X$, for lack
of a suitable reference in the literature.

Throughout this paper, if $R$ is a commutative ring, $\Frac(R) =
T^{-1}R$ denotes \emph{the total fraction ring} of $R$, where $T
\subset R$ is the set of non-zerodivisors of $R$.

\begin{defn}
A \emph{rational function on $X$} is a rule $f$ that associates to
each affinoid subdomain $U = \Max(A)$ of $X$ an element $f_U \in
\Frac(A)$, in a manner compatible with restriction morphisms for
inclusions of affinoid subdomains.  We write $\kappa(X)$ for the
$L$-algebra of such $f$.
\end{defn}

One could equivalently define $\kappa(X)$ to be the categorical limit
of the $L$-algebras $\Frac(A)$ as $\Max(A)$ ranges over all affinoid
subdomains inside $X$.  Note that if $X = \Max(A)$ is affinoid then
$\kappa(X) = \Frac(A)$, because $X$ is the final affinoid subdomain of
$X$.

The following first results suffice for our purposes.

\begin{prop}\label{P:rational functions}
The $L$-algebra of rational functions has the following properties.

(1) The formation of $\kappa(X)$ is functorial for pullback under
dominant morphisms.

(2) If $U$ is a dense Zariski open in $X$, then the map $\kappa(X) \to
\kappa(U)$ is injective.

(3) Assigning to each affinoid subdomain $U = \Max(A)$ the $L$-algebra
$\Frac(A)$ defines a sheaf $\mathscr M_X$ on $X$ for the weak topology
on $X$ (generated by affinoid subdomains).  In particular, $\kappa(X)
= \Gamma(X, \mathscr M_X)$.

(4) If $\pi \cn \wt X \to X$ is a normalization, then $\pi_*(\scrM_\wt
X) = \scrM_X$ and $\kappa(X) =
\kappa(\wt X)$.

(5) If $\{X_i\}_{i \in I}$ is the set of irreducible components of
$X$, then $\kappa(X) = \prod_{i \in I} \kappa(X_i)$.

(6) If $X$ is irreducible, then $\kappa(X)$ is a field.
\end{prop}

\begin{proof}
The claims (1) and (2) are obvious.

The claim (3) is essentially \cite[Theorem~4.6.6]{FvdP}; the theorem
therein was stated for the covers by rational subdomains because the
Gerritzen--Grauert theorem only appeared later, but the proof works
verbatim for the weak topology.

When $X$ is affinoid it is obvious that $\kappa(X) = \kappa(\wt X)$.
Passing to the construction of (3), this gives $\pi_*(\scrM_\wt X) =
\scrM_X$ in general, whence taking global sections of the latter
identity yields $\kappa(X) = \kappa(\wt X)$ in general.  This shows
(4).

We treat (5).  The irreducible components $X_i$ of $X$ are defined to
be the images of the connected components $\wt X_i$ of $\wt X$.
Moreover, each $\wt X_i \to X_i$ is a normalization.  Thus (4) reduces
us to the case where $X$ is normal, and $\{X_i\}_{i \in I}$ is none
other than the set of connected components of $X$, in which case the
claim is obvious.

In the setting of (6), we immediately reduce to the case where $X$ is
normal.  Then each affinoid subdomain $U = \Max(A)$ of $X$ is a finite
disjoint union of connected affinoid subdomains $U_i = \Max(A_i)$, and
each $\Frac(A_i)$ is a field.  Thus we are reduced to showing that if
$f \in \kappa(X)$ is not the zero element then, for each connected
affinoid subdomain $U$ of $X$, the element $f_U \in \Frac(U)$ is
nonzero.  By definition, the fact that $X$ is connected means that for
any two points $P,Q \in X$, there exists a chain of connected affinoid
subdomains $U_1,\ldots,U_N$ of $X$ with $P \in U_1$, $Q \in U_N$, and
$U_i \cap U_{i+1}$ nonempty.  Given nonzero $f$, choose $V$ such that
$f_V$ is nonzero, and arbitrarily choose $P \in V$.  Given any $U$,
choose arbitrary $Q \in U$, and connect $P,Q$ with a chain
$U_1,\ldots,U_N$ as in the definition of connectedness.  Insert $V$ at
the beginning of the chain, and $U$ at the end of the chain, so that
now $U_1 = V$ and $U_N = U$.  It suffices to see that if $f_{U_i}$ is
nonzero then so is $f_{U_{i+1}}$.  Indeed, choose any connected
affinoid subdomain $W_i$ contained in $U_i \cap U_{i+1}$, write $U_i =
\Max(A_i)$ and $W_i = \Max(B_i)$, and note that the maps
\[
p_i \cn \Frac(A_i) \to \Frac(B_i),
\qquad
p'_i \cn \Frac(A_{i+1}) \to \Frac(B_i)
\]
are homomorphisms of fields, hence injective, and $p_i(f_{U_i}) =
f_{W_i} = p'_i(f_{U_{i+1}})$.
\end{proof}

\begin{caution}
If $U$ is a dense Zariski open in $X$, the natural map $\kappa(X) \to
\kappa(U)$ is not in general surjective.  For example, when $X =
\Max(\Qp\langle z\rangle)$ and $U = X \backslash \{0\}$, the function
$f(z) = \prod_{n \geq 1} \big( 1+ \frac{p^n}{z-p^n}\big)$ belongs to
$\kappa(U)$ but not to $\kappa(X)$.
\end{caution}

Assume the characteristic of $L$ is not two, and write $\{X_i\}_{i \in
  I}$ for the set of irreducible components of $X$.  We say that $f
\in \kappa(X)$ \emph{has values in $\{\pm1\}$} if it belongs to the
subgroup $\prod_{i \in I} \{\pm1\}$ of $\prod_{i \in I}
\kappa(X_i)^\times = \kappa(X)^\times$.

\subsection{Local arguments: $\ell \neq p$}

Here we let $K$ be as in the start of this section, and we assume that
the characteristic $\ell$ of $k$ is not equal to $p$.

\begin{prop}\label{P:WD in families}
Let $X$ be a reduced rigid analytic space over $\Qp$, and let $\scrT$
be a locally free coherent $\calO_X$-module equipped with a
continuous, $\calO_X$-linear action of $G_K$.  Then the usual recipe
(see, for example, \cite[(5.1.4)]{N3}) associates to $\scrT$ a
Weil--Deligne representation structure $WD(\scrT)$ on $\scrT$, in a
manner that is functorial in $\scrT$ and commutes with arbitrary base
change in $X$.
\end{prop}

\begin{proof}
By gluing it suffices to treat the case where $X = \Max(A)$ is
affinoid, and in this case it suffices to know the claim that the wild
inertia subgroup $I_K' \subset G_K$ acts on $\scrT$ through a finite
quotient.  To see this, \cite[Lemme~3.18]{Ch} shows that there exists
a unit ball subalgebra $A_0 \subset A$ and a finite, flat,
$G_K$-stable $A_0$-lattice $\scrT_0 \subset \scrT$.  By replacing
$\scrT_0$ by its direct sum with another finite flat $A_0$-module with
trivial $G_K$-action, we may assume it is free, say of rank $d$.  A
cofinal system of open subgroups of $\Aut_{A_0}(\scrT_0)$ is given by
$1+p^n\End_{A_0}(\scrT_0) \approx 1+p^nM_d(A_0)$, where $d$ is the
rank of $\scrT_0$.  Note that $(1+p^nM_d(A_0))/(1+p^{n+1}M_d(A_0))
\cong M_d(A_0/p)$ for $n \gg 0$, where the right hand side is
annihilated by $p$.  Since $I_K'$ is a pro-$\ell$ group, this implies
that the order of its image in $\Aut_{A_0}(\scrT_0/p^n)$ stabililizes
for $n \gg 0$, and therefore its image in $\Aut_{A_0}(\scrT_0)$ is
finite, as was desired.
\end{proof}

\begin{notation}
Define a \emph{local datum} $\scrX = (X,\scrT,j)$ \emph{(for $K$)} to
consist of:
\begin{itemize}
\item a reduced rigid analytic space $X$ over $\Qp$,
\item a locally free coherent $\calO_X$-module of rank $d$ equipped
  with a continuous, linear $\calO_X$-action of $G_K$, and
\item a skew-symmetric, $\calO_X[G_K]$-linear morphism
  $\langle\cdot,\cdot\rangle_j \cn \scrT \otimes_{\calO_X} \scrT \to
  \calO_X(1)$ whose adjoint map $j \cn \scrT \to \scrT^*(1)$ is
  generically on $X$ an isomorphism (so $d$ is even).
\end{itemize}
If $\scrX$ is a local datum as above and $f \cn Y \to X$ is a morphism
of reduced rigid analytic spaces, then the tuple $f^*\scrX =
(Y,f^*\scrT,f^*j)$ is again a local datum if and only if $f^*j$ is
generically on $Y$ an isomorphism.
\end{notation}

We will associate to a local datum two subsets $X_\alg \subseteq
X_\pair \subseteq X$.  First, $X_\pair$ is the set of points $P \in X$
for which there exists $u \in \Frac(\calO_{X,P}^\wedge)^\times$ such
that $uj_P^\wedge \cn \scrT_P^\wedge \to \scrT^*(1)_P^\wedge$ is an
isomorphism of $\calO_{X,P}^\wedge$-modules.  Equivalently, by
Nakayama's lemma, $uj_P^\wedge$ sends $\scrT_P^\wedge$ into
$\scrT^*(1)_P^\wedge$ and $\overline{uj} = uj_P^\wedge
\otimes_{\calO_{X,P}^\wedge} \kappa(P) \cn \scrT \otimes_{\calO_X}
\kappa(P) \to \scrT^*(1) \otimes_{\calO_X} \kappa(P)$ is an
isomorphism.  We say that such $u$ \emph{guarantees} $P \in X_\pair$.
For $u \in \kappa(X)^\times$, we write $X_\pair^u$ for the set of $P
\in X$ such that the image of $u$ in
$\Frac(\calO_{X,P}^\wedge)^\times$ guarantees $P \in X_\pair$.  If $f
\cn Y \to X$ is a morphism such that $f^*\scrX$ is again a local
datum then one has $f^{-1}(X_\pair) \subseteq Y_\pair$, and if
moreover $f^*(u)$ is well-defined then $f^{-1}(X_\pair^u) \subseteq
Y_\pair^{f^*u}$.

\begin{prop}\label{P:u over affinoid}
Let $\scrX$ be a local datum.

(1) If $\iota \cn Y \subseteq X$ is an admissible open subset, then
$\iota^*\scrX$ is a local datum and satisfies $Y_\pair = Y \cap
X_\pair$.

(2) If $X = \Max(A)$ is affinoid, then for each $P \in X_\pair$ there
exists $u \in \Frac(A)^\times$ such that $P \in X_\pair^u$.

(3) For each $u \in \kappa(X)^\times$ the subset $X_\pair^u \subseteq
X$ is a dense Zariski open subspace.  In particular, $X_\pair$ is also a dense
Zariski open subspace.
\end{prop}

\begin{proof}
Claim (1) is clear from the definitions of a local datum and
$X_\pair$.

For (2), since $\Frac(A)$ depends only on the normalization of $A$, it
suffices to assume $A$ is normal, and therefore also irreducible.
Choose $u_1 \in \Frac(\calO_{X,P}^\wedge)^\times$ guaranteeing $P \in
X_\perf$, and write $u_1 = s_1^{-1}f_1$ with $f_1,s_1 \in
\calO_{X,P}^\wedge$ and $s_1$ a non-zerodivisor.  By the definition of
$\calO_{X,P}$, we may find an affinoid subdomain $\Max(A') \subseteq
\Max(A)$ containing $P$ and $f_2,s_2 \in A$ such that $f_2 \equiv f_1$
and $s_2 \equiv s_1$ modulo a high power of the maximal ideal at $P$.
We may take $\Max(A')$ to be connected, and therefore also irreducible
because $A$ is normal.  It follows that $s_2$ is a non-zerodivisor.
But the morphism $A \to A'$ is injective with dense image for the
spectral norm $|\cdot|_{A'}$ on $A'$, so we may choose $f_3,s_3 \in A$
with $|f_3-f_2|_{A'} < |f_2|_{A'}$ and $|s_3^{\pm1}-s_2^{\pm1}|_{A'} <
|s_2^{\pm1}|_{A'}$.  It is easy to check that $u = s_3^{-1}f_3 \in
\Frac(A)$ satisfies $P \in X_\perf^u$.

Consider (3).  We may assume $X = \Max(A)$ is affinoid.  By Nakayama's
lemma, the conditions defining membership $P \in X_\pair^u$ in terms
of $uj$ are Zariski open.  On the other hand, writing $u = s^{-1}f$
with $f,s \in A$ and $s$ a non-zerodivisor, away from the union of the
vanishing loci of $f,s$ one has $X_\perf^u = X_\perf^1$, and by
definition of a local datum the right hand side is dense in $X$.  The
final claim also follows, because part (2) shows that $X_\pair =
\bigcup_{u \in \Frac(A)^\times} X_\pair^u$.
\end{proof}

Given a datum $\scrX$ as above with $X$ irreducible, it follows from
the preceding proposition that the Weil--Deligne representation
$WD(\scrT) \otimes_{\calO_X} \kappa(X)$ is symplectic self-dual, and
then \cite[Proposition~2.2.1(2)]{N3} gives $\vep(WD(\scrT)
\otimes_{\calO_X} \kappa(X)) \in \{\pm1\}$.  If we do not assume $X$
to be irreducible, working one irreducible component at a time gives
an element of $\kappa(X)^\times$ with values in $\{\pm1\}$.
Similarly, if $P \in X_\pair^u$ then $\overline{uj}$ makes $WD(\scrT)
\otimes_{\calO_X} \kappa(P)$ symplectic self-dual, so that
$\vep(WD(\scrT) \otimes_{\calO_X} \kappa(P)) \in \{\pm1\}$.

Finally, we define $X_\alg$ to be $X_\pair$ if $\scrT$ is unramified,
and otherwise to be the set of points $P \in X_\pair$ such that $\scrT
\otimes_{\calO_X} \kappa(P)$ is pure of weight $-1$.

\begin{prop}\label{P:local eps locally constant}
Suppose $\scrX$ is a local datum with $X$ irreducible and $P \in
X_\alg$.  Then one has $\vep(\scrT \otimes_{\calO_X} \kappa(P)) =
\vep(\scrT \otimes_{\calO_X} \kappa(X))$.  Moreover, $\scrT
\otimes_{\calO_X} \kappa(P)$ is ramified if and only if $\scrT$ is.
\end{prop}

\begin{proof}
The first claim is trivial if $\scrT$ is unramified, because then both
sides of the desired identity are $1$, so we assume that $\scrT$ is
ramified.

Write, for brevity, $\scrV = \scrT \otimes_{\calO_X} \kappa(X)$ and $V
= \scrT \otimes_{\calO_X} \kappa(P)$.  In this proof we will have no
need for the underlying $G_K$-representation structures on
$\scrT,\scrV,V$, so we use the latter symbols to denote these objects
only \emph{equipped with their Weil--Deligne representation
  structures}.  I.e., we identify $WD(\scrT) = \scrT$, $WD(\scrV) =
\scrV$, and $WD(V) = V$.

We begin with some reductions using implicitly, and repeatedly, the
fact that formation of $\vep$-factors commutes with extension of
coefficient fields.  First, we may assume that $X$ is normal, and it
suffices to replace $X$ by an arbitrary small connected affinoid
subdomain $\Max(A)$ containing $P$.  By Proposition~\ref{P:u over
  affinoid}(2) we may choose $u$ such that $P \in X_\pair^u$, and by
Proposition~\ref{P:u over affinoid}(3) by replacing $j$ by $uj$ and
shrinking $X$ we may assume $j$ is an isomorphism.  Further shrinking
$X$, we may assume $\scrT$ is free.  Then blowing up $X$ along the
$(\rank_A \coker(N^r \mid \scrT))$th Fitting ideals of each of the
$\coker(N^r \mid \scrT)$ (see, for example, \cite[\S6.3]{KPX}), and
again replacing $X$ by a small neighborhood of a preimage of $P$, we
may assume that each of the $\img(N^r \mid \scrT)$ are flat and
moreover free.

Letting $S$ be a free polynomial variable, we may replace $A$ by a
finite integral extension over which the characteristic polynomial
$Q(S) = \det(S-f \mid \scrT)$ splits, and $P$ with any maximal ideal
of this extension lying over it.  In particular, the reduction $\bar
Q(S) = \det(S-f \mid V)$ splits over $\kappa(P)$.  If we write $V_i$
for the maximal $f$-stable subspace whose eigenvalues are strictly
pure of weight $i-1$, then the $\bar Q_i(S) = \det(S-f \mid V_i)$ for different $i \in \ZZ$ are
coprime in $\kappa(P)[S]$.  Because $V$ is pure of weight $-1$ we have
$V = \oplus_{i \in \ZZ} V_i$, and the factorization $\bar Q = \prod_i
\bar Q_i$.  There is a unique lifting of this factorization to $Q =
\prod Q_i$ in $A[S]$ with the $Q_i$ monic: $Q_i$ is characterized as
the largest monic divisor of $Q$ all of whose roots have image in
$\kappa(P)$ strictly pure of weight $i-1$.  Since $\bar Q_i,\bar Q_j$
generate the unit ideal in $\kappa(P)[S]$ if $i \neq j$, after perhaps
shrinking $X$ around $P$ one has that $Q_i,Q_j$ generate the unit
ideal in $A[S]$ if $i \neq j$.  It follows that the decomposition $V =
\oplus_{i \in \ZZ} V_i$ is obtained by specializing at $P$ the
decomposition $\scrT = \bigoplus_{i \in \ZZ} \scrT_i$, where $\scrT_i
= \ker(Q_i(f) \mid \scrT)$.  We set $\scrV_i = \scrV \otimes_A
\kappa(X)$ to obtain a decomposition $\scrV = \oplus_{i \in \ZZ}
\scrV_i$.  The $\scrT_i$ are projective by construction, and one has
$\det(S-f \mid \scrT_i) = \det(S-f \mid \scrV_i) = Q_i(S)$.

Since $\scrV_i$ is the subspace of $\scrV$ spanned by those
generalized $f$-eigenvectors whose eigenvalues are roots of $Q_i(S)$,
it follows from the characterization of the $Q_i$ above that
$N(\scrV_i) \subseteq \scrV_{i-2}$.  Because $\scrT_i = \scrV_i \cap
\scrT$, this implies that also $N(\scrT_i) \subseteq \scrT_{i-2}$.
For $i \geq 0$, the purity of $V$ (because $P \in X_\mathrm{alg}$) implies that the map $N^i \cn V_i
\to V_{-i}$ is an ismorphism, so $N^i \cn \scrT_i \to \scrT_{-i}$
becomes isomorphism after perhaps shrinking $X$ around $P$, and
therefore $N^i \cn \scrV_i \to \scrV_{-i}$ is also an isomorphism.  It
follows from the characterization of the monodromy filtration given in
\cite[(1.3.1)]{N3} that $\{\scrV_i\}_i$ is a splitting of the
monodromy filtration on $\scrV$.  We conclude that, in the notations
of \cite[(1.3.6)]{N3}, one has
\[
d_i(\scrV) = \dim_{\kappa(X)} \scrV_i = \rank_A \scrT_i =
\dim_{\kappa(P)} V_i = d_i(V),
\]
and therefore also $m_i(\scrV) = m_i(\scrT)$ and $\rank (N^r \mid
\scrV) = \rank(N^r \mid V)$ for $r \geq 0$.  It is easy to see that
the natural map
\begin{equation}\label{E:control N}
\img(N^r \mid \scrT) \otimes_A \kappa(P) \to \img(N^r \mid V)
\end{equation}
is surjective.  However, we have arranged that $\img(N^r \mid \scrT)$
is a projective $A$-module, so the preceding numerics show that the
source and target of the above map have the same
$\kappa(P)$-dimension.  It follows that \eqref{E:control N} is an isomorphism.

We return to $\vep$-factors.  One has factorizations
\[
\vep(\scrV)
 = \vep(\scrV^{N\text{-ss}}) \det(-f \mid \scrV_g/\scrV_f),
\qquad
\vep(V) = \vep(V^{N\text{-ss}}) \det(-f \mid V_g/V_f)
\]
using \cite[(2.1.2.3)]{N3}; in each equation the first two terms
belong to $\{\pm1\}$ and therefore the third does too.  It suffices to
compare the factors on the right hand sides.  An argument is given to
treat each factor in the proof of \cite[Proposition~2.2.4]{N3}, and
these arguments apply directly to our setting, granted one knows the
following two facts.  The first fact is that the natural maps
\[
\scrT^{N\text{-ss}} \otimes_A \kappa(P)
\stackrel\sim\to
V^{N\text{-ss}},
\qquad
\scrT_g/\scrT_f \otimes_A \kappa(P)
\stackrel\sim\to
V_g/V_f
\]
are isomorphisms.  The first is obvious, and the second follows from
the isomorphism \eqref{E:control N} and the fact that $N$ induces an
isomorphism
\begin{equation}\label{E:compare N}
\scrT_g/\scrT_f \stackrel\sim\to \img(N \mid \scrT)^{I_K}.
\end{equation}
The second fact is that the module $\scrT$ (resp.\ $\scrT_g/\scrT_f$)
is flat, so that one can form the trace (resp.\ determinant) of an
endomorphism in a manner compatible with base change.  But for $\scrT$
(resp.\ for $\scrT_g/\scrT_f$), this is obvious (resp.\ follows from
the isomorphism \eqref{E:compare N} and passing to a direct summand).

We now treat the second claim, continuing with the notations just
introduced, and in particular considering $\scrT$ and $V$ as
Weil--Deligne representations.  It suffices to show that inertia
(resp.\ monodromy) acts trivially on $\scrT$ if and only if it does on
$V$.  For inertia, this follows from the fact that the unique
$I_K$-stable decomposition $\scrT = \scrT^{I_K} \oplus \scrT'$ is
compatible with arbitrary base change.  For monodromy, this becomes
clear after we may make the same blowups and shrinkings of $X$ around
$P$ as in the proof of the first claim, to obtain the isomorphism
\eqref{E:control N} with $\img(N^r \mid \scrT)$ a projective
$A$-module.
\end{proof}

We associate to the local datum $\scrX$ the complex
$C^\bullet(G_K,\scrT)$ of continuous cochains of $G_K$ on $\scrT$, and
also the complex $C^\bullet_\unr(G_K,\scrT) = [\scrT^{I_K}
  \xrightarrow{f-1} \scrT^{I_K}]$ in degrees $0,1$.

\begin{prop}\label{P:acyclic}
For each $P \in X_\pair$ for which $WD(V)$ is pure, after perhaps
shrinking $X$ around $P$, the complexes $C^\bullet(G_K,\scrT)$ and
$C^\bullet_\unr(G_K,\scrT)$ are acyclic.
\end{prop}

\begin{proof}
It suffices to work in a sufficiently small affinoid neighborhood
$\Max(A)$ of $P$.  For the first complex, the argument of
\cite[Lemma~4.1.5]{KPX} shows that it suffices to check that
$C^\bullet(G_K,\scrT) \Lotimes_A \kappa(P)$ is acyclic; this claim
follows from the base change result of \cite[Theorem~1.4(1)]{P}
together with \cite[4.2.2(1)]{N3}.  For the second complex, we take
mapping fibers of $f-1$ on the distinguished triangle
\[
\scrT^{I_K} \to C^\bullet(I_K,\scrT) \to \scrT_{I_K}[-1]
\]
to obtain the distinguished triangle
\begin{equation}
\label{E:unr to total to coinv}
C^\bullet_\unr(G_K,\scrT)
\to
C^\bullet(G_K,\scrT)
\to
[\scrT_{I_K} \xrightarrow{f-1} \scrT_{I_K}],
\end{equation}
with the last complex concentrated in degrees $1,2$.  We compute
\[
\scrT_{I_K}/(f-1) \otimes_A \kappa(P)
\cong (\scrT \otimes_A \kappa(P))_{G_K}
= 0
\]
by the purity of $\scrT \otimes_A \kappa(P)$.  In particular, after
shrinking $X$ around $P$ we have that $f-1$ is surjective, hence
bijective, on $\scrT_{I_K}$.  This forces $[\scrT_{I_K}
  \xrightarrow{f-1} \scrT_{I_K}]$ to be acyclic, and so is
$C^\bullet_\unr(G_K,\scrT)$ by \eqref{E:unr to total to coinv}.
\end{proof}

\subsection{Global data}\label{SS:global data}

The rest of this section is concerned with adapting
\cite[\S\S4--5]{N3} to the setting of rigid analytic spaces.  In order
to state our main theorem, we will first need to invoke a good deal of
notation, consisting mostly of straightforward modifications of
\cite[\S\S4--5]{N3}.  \emph{These notations will be in force through
  the remainder of this section.}

We refer to \cite{KPX} for definitions and background on
(``arithmetic'') families of $(\vphi,\Ga)$-modules over rigid analytic
spaces.

For the rest of this section, $F$ is a number field, $S$ is a fixed
finite set of places of $F$ containing the places $S_p$ above $p$ and
the archimedean places $S_\infty$, $F_S$ is a maximal algebraic
extension of $F$ unramified outside $S$, and $G_{F,S} = \Gal(F_S/S)$.
For a place $v$ of $F$ we fix an algebraic closure $F_v^\alg$ of $F_v$
and an embedding $F_S \hookrightarrow F_v^\alg$, giving maps $G_{F_v}
\to G_{F,S}$.  If $X$ has an action of $G_{F,S}$, we write $X_v$ for
$X$ with the action restricted to $G_{F_v}$.

\begin{defn}\label{D:global datum}
Define a \emph{global point datum} $(L,V,j,\{S_v\}_{v \in S_p})$ to
consist of:
\begin{itemize}
\item a finite extension $L/\Qp$,
\item a finite-dimensional $L$-vector space $V$ of dimension $d$
  equipped with a continuous, $L$-linear action of $G_{F,S}$,
\item a skew-symmetric, $L[G_{F,S}]$-linear morphism
  $\langle\cdot,\cdot\rangle \cn V \otimes_L V \to L(1)$ whose adjoint
  map $j: V \to V^*(1)$ is an isomorphism,
\item for each $v \in S_p$, a short exact sequence
\[
S_v \cn 0 \to D_v^+ \to D_v \to D_v^- \to 0
\]
  of $(\vphi, \Ga)$-modules over $\calR_L(\pi_{F_v})$, where $D_v =
  \bbD_\rig(V_v)$, such that $\langle
  D_v^+,D_v^+\rangle_{\bbD_\rig(j)} = 0$, $\rank_{\calR_L(\pi_{F_v})}
  D_v^\pm = d/2$, and $S_v$ makes $D_v$ Panchishkin, and
\item for every finite place $v$ where $V_v$ is ramified, $V_v$ is
  pure (necessarily of weight $-1$).
\end{itemize}
Usually we will write simply $V$ for the tuple.

Define a \emph{global datum} $\scrX = (X,\scrT,j,\{\scrS_v\}_{v \in
  S_p})$ to consist of:
\begin{itemize}
\item a reduced rigid analytic space $X$ over $\Qp$,
\item a locally free coherent $\calO_X$-module $\scrT$ of rank $d$
  equipped with a continuous, $\calO_X$-linear action of $G_{F,S}$,
\item a skew-symmetric, $\calO_X[G_{F,S}]$-linear morphism
  $\langle\cdot,\cdot\rangle_j \cn \scrT \otimes_{\calO_X} \scrT \to
  \calO_X(1)$ whose adjoint map $j \cn \scrT \to \scrT^*(1)$ is
  generically on $X$ an isomorphism (so $d$ is even), and
\item for each $v \in S_p$ a short exact sequence
\[
\scrS_v \cn 0 \to \scrD_v^+ \to \scrD_v \to \scrD_v^- \to 0
\]
  of $(\vphi,\Ga)$-modules over $\calR_{\calO_X}(\pi_{F_v})$, where
  $\scrD_v = \bbD_\rig(\scrT_v)$, such that
  $\langle\scrD_v^+,\scrD_v^+\rangle_{\bbD_\rig(j)} = 0$ and
  $\rank_{\calR_{\calO_X}(\pi_{F_v})} \scrD_v^\pm = d/2$.
\end{itemize}
Note that for any place $v \notin S_p \cup S_\infty$, the tuple
$\scrX_v = (X,\scrT_v,j)$ is a local datum for $F_v$.

A global point datum $(L,V,j,\{S_v\}_{v \in S_p})$ determines a global
datum upon taking $X = \{P\}$ to be the spectrum of $L$.  Conversely,
when $P \in X$ is fixed we write $L = \kappa(P)$, $V = \scrT
\otimes_{\calO_X} \kappa(P)$, $\ov j = j \otimes_{\calO_X} \kappa(P)$,
$S_v = \scrS_v \otimes_{\calO_X} \kappa(P)$, and $D_v^? = \scrD_v^?
\otimes_{\calO_X} \kappa(P)$ for $?=\emptyset,\pm$.  In the sequel we
will specifiy some points $P \in X$ for which $(L,V,\ov j,\{S_v\}_{v
  \in S_p})$ is a global point datum.
\end{defn}

Global data admit the same functorialities as local data, and the
notations $X_\pair$ and $X_\pair^u$ and Proposition~\ref{P:u over
  affinoid} carry over changing only ``local'' to ``global''
throughout.  In particular, when $X = \{P\}$ arises from a global
point datum, one has $X = X_\pair^1$.

We have placed no conditions on points $P$ of $X_\pair$ concerning the
duality of $D_v^\pm$, because the next proposition shows the duality
to be automatic.

\begin{prop}
\label{P:uj gurantees Drig}
If $u$ guarantees $P \in X_\pair$, then $\bbD_\rig(\overline{uj})$
induces, for each $v \in S_p$, isomorphisms $D_v^\pm \stackrel\sim\to
(D_v^\mp)^*(1)$.
\end{prop}

\begin{proof}
Since the map $D_v^\pm \to (D_v^\mp)^*(1)$ has full rank by the
nondegeneracy of $uj$, we just need to check that it has saturated
image, or equivalently that the top exterior power $\bigwedge^{d/2}
D_v^\pm \to \bigwedge^{d/2} (D_v^\mp)^*(1)$ has saturated image.
Indeed, one has an identification
\[
(\bigwedge^{d/2} D_v^\pm)
  \otimes_{\calR_{\kappa(P)}(\pi_{F_v})}
  (\bigwedge^{d/2} D_v^\mp)
\cong
\bigwedge^d D_v
=
\calR_{\kappa(P)}(\pi_{F_v})(d/2)
\]
arising from the sequence $S_v$ and the identification $\bigwedge^d V
\cong \kappa(P)(d/2)$ of \cite[(1.2.2)(2)]{N3}, hence there exists an
isomorphism $\bigwedge^{d/2} D_v^\pm \cong (\bigwedge^{d/2}
D_v^\mp)^*(d/2) = \bigwedge^{d/2} (D_v^\mp)^*(1)$, and any nonzero
endomorphism of a rank one object is an isomorphism and in particular
has saturated image.  (One can check that the two maps agree up to
sign, but we do not need this fact.)
\end{proof}

\begin{notation}
We define $X_\alg$ to be the set of points $P \in X_\pair$ satisfying
\begin{itemize}
\item for each $v \in S_p$, the sequence $S_v$ makes $D_v$
  Panchishkin, and
\item for each $v \notin S_\infty$ such that $\scrT|_{G_{F_v}}$ is
  ramified, $V_v$ is pure (of weight $-1$).
\end{itemize}
We make some remarks on this notation.  First, the purity hypothesis
always applies in particular to $v \in S_p$, but only applies to $v$
belonging to $S$.  Second, for $v \notin S_p \cup S_\infty$,
Proposition~\ref{P:local eps locally constant} shows that $V_v$ is
ramified if and only if $\scrT_v$ is ramified.  Thus, if $P \in X_\alg
\cap X_\pair^u$, then $(\kappa(P),V,\ov{uj},\{S_v\}_{v \in S_p})$ is a
global point datum.  The choice of $u$ will not matter, so we
abusively use $V$ to denote this tuple.  Third, if $f \cn Y \to X$ is
a morphism such that $f^*\scrX$ is a global datum, then one has
$f^{-1}(X_\alg) \subseteq Y_\alg$.  Finally, for any place $v \notin
S_p \cup S_\infty$, the subset $X_\alg$ for $\scrX$ is contained in
the subset $X_\alg$ for $\scrX_v$.
\end{notation}

\begin{example}\label{Ex:nekovar}
The data $(R, \mathcal{T}, (\cdot,\cdot), \{\mathcal{T}_v^+\}_{v \in
  S_p},P,u)$, assumed given in \cite[(5.1.2)]{N3} give rise to a
global datum $\scrX = (\mathrm{Spf}(R,\gothm_R)^\an, \mathcal{T}^\an,
(\cdot,\cdot)^\an, \{\bbD_\rig(\mathcal{T}_v^{+,\an})\}_{v \in S_p})$
and a point $P \in X_\alg \cap X_\pair^u$.
\end{example}

\subsection{Global algebraic invariants}

In this subsection we treat the variation of Selmer groups in
families.  We continue with the notations and hypotheses of
Subsection~\ref{SS:global data}.

Let $T$ be either $\scrT$ where $\scrX = (X,\scrT,j,\{\scrS_v\}_{v \in
  S_p})$ is a global datum, or $V$ where $(L,V,j,\{S_v\}_{v \in S_p})$
is a global point datum.  For $v \in S_p$ write $\bbD_\rig(T_v)^+ =
\scrD_v^+,D_v^+$ (respectively).  Choose a set $\Sigma$ with $S_p
\subseteq \Sigma \subseteq S \bs S_\infty$ of primes at which $T$ is
ramified.  Then the \emph{Selmer complex} (cf.\ \cite[\S5.2]{N3}) is
defined by
\[
\widetilde C^\bullet_f(G_{F,S},T;\Sigma) =
\Fib\left(
  C^\bullet(G_{F,S},T) \oplus \bigoplus_{v \in S-S_\infty} U_v^+(T)
  \longrightarrow \bigoplus_{v \in S-S_\infty} C^\bullet(G_{F_v},T)
\right),
\]
where $\Fib(f) = \Cone(f)[-1]$ denotes the mapping fiber of a morphism
of complexes, and
\[
U_v^+(T) = \begin{cases}
C^\bullet_{\vphi,\ga}(\bbD_\rig(T_v)^+), & v \in S_p, \\
0, & v \in \Sigma-S_p, \\
C^\bullet_\unr(G_{F_v},T)
 = [T^{I_v} \xrightarrow{f_v-1} T^{I_v}], & v \in S-\Sigma.
\end{cases}
\]
In the case $T = \scrT$, it is easy to see that the formation of these
complexes commutes (up to canonical quasi-isomorphism) with base
change to open subsets of $X$.  In particular, their cohomologies form
coherent sheaves on $X$.

When the choice of $S$ and $\Sigma$ is not important we will write
$\widetilde C^\bullet_f(F,T)$ for the Selmer complex, $\widetilde
H^\bullet_f(F,T)$ for its cohomology, and $\widetilde h^1_f(F,T)$ for
its generic rank when $T=\scrT$ (resp.\ dimension when $T=V$).  By the
following proposition, which is a straightforward adaptation of
\cite[Proposition~4.2.2]{N3} and \cite[Proposition~5.2.2(2)]{N3} to
our setup, the rational function (resp.\ integer) $\wt h^1_f(F,T)$ is
independent of choice of $S$ and $\Sigma$.

For $V$ a global point datum, we use $H^1_f(F,V)$ to denote the usual
\emph{Bloch--Kato Selmer group} (see, for example,
\cite[(4.1.1)]{N3}), and set $h^1_f(F,V) = \dim_L H^1_f(F,V)$.

\begin{prop}\label{P:selmer complex}
(1) Suppose $(X,\scrT,j,\{\scrS_v\}_{v \in S_p})$ is a global datum.
  For each $P \in X_\alg$, after perhaps shrinking $X$ around $P$, the
  complex $\widetilde C^\bullet_f(G_{F,S},\scrT;\Sigma)$ is
  canonically independent (up to quasi-isomorphism) of $S$ and
  $\Sigma$.

(2) Suppose $(L,V,j,\{S_v\}_{v \in S_p})$ is a global point datum.
  The complex $\widetilde C^\bullet_f(G_{F,S},V;\Sigma)$ is
  canonically independent of (up to quasi-isomorphism) of $S$ and
  $\Sigma$.  Moreover, we have $\wt h^1_f(F,V) = h^1_f(F,V) + \sum_{v
    \in S_p} h^0(D_v^-)$.
\end{prop}

\begin{proof}
In both (1) and (2), the independence with respect to $\Sigma$ follows
from Proposition~\ref{P:acyclic}, and the independence with respect to
$S$ is a general fact \cite[Proposition~7.8.8]{N-selmer}.  We take
$\Sigma=S_p$ from now on. By the definition of the Selmer complex as a
mapping fiber and Lemma~\ref{L:panchishkin H1f}, we deduce an exact
sequence
\[
H^0(G_{F,S},V)
\to \oplus_{v \in S_p} H^0(D_v^-)
\to \widetilde H^1_f(F, V)
\to H^1_f(F, V)
\to 0.
\]
The purity of $V$ forces $H^0(G_{F,S},V)=0$, and the desired identity
follows.
\end{proof}

We have seen that the ability to use the choice $\Sigma = S_p$ allows
one to compare $\wt H^1_f(F,V)$ to $H^1_f(F,V)$.  On the other hand,
taking $\Sigma$ to be the set of all ramified primes for $\scrT$
allows one to use general base-change results without complications,
which is an ingredient in the proof of Proposition~\ref{P:vep_alg}
below.

\begin{prop}\label{P:vep_alg}
If $\scrX$ is a global datum with $X$ irreducible and $P \in X_\alg$,
then one has $\wt h^1_f(F,V) \equiv \wt h^1_f(F,\scrT) \pmod 2$.
\end{prop}

\begin{proof}
We use implicitly, and repeatedly, the fact that the formation of $\wt
h^1_f(F,\scrT)$ commutes replacing $X$ by an admissible open subset,
or by any space $Y$ with a morphism $f \cn Y \to X$ that is
generically an isomorphism (although we make no claim about whether
$\wt C^\bullet_f(F,\scrT)$ itself commutes with the latter base
change).  In particular, if $u$ guarantees $P \in X_\pair$ then we may
replace $j$ by $uj$ and shrink $X$ around $P$ so that $j$ is an
isomorphism, and moreover $\scrT$ is free.  We may also replace $X$
with a resolution of singularities (and $P$ by a point in its
preimage) and, after a constant scalar extension, shrink $X$ to have
the form $X = \Max(A)$ with $A = \kappa(P)\langle
T_1,\ldots,T_r\rangle$, with $P$ corresponding to $T_1 = \ldots = T_r
= 0$.

Define the increasing system of closed subspaces $\{P\} = X_0 \subset
X_1 \subset \cdots \subset X_r = X$ by letting $X_j$ correspond to
$T_{j+1} = \cdots = T_r = 0$ with inclusion $\iota_j \cn X_j
\hookrightarrow X$, so that $X_j = \Max(A_j)$ with $A_j =
\kappa(P)\langle T_1,\ldots,T_j\rangle$.  Write $\scrT_j =
\iota_j^*\scrT$ for brevity.  Note that each $\scrX_j =
\iota_j^*\scrX$ is a global datum with $X_{j,\pair} = X_j$ and $P \in
X_{j,\alg}$.

After shrinking $X$ (and therefore the $X_j$) around $P$ by rescaling
the parameters $T_j$, we may assume (Proposition~\ref{P:acyclic}) that
all the complexes $C^\bullet(G_{F_v},\scrT_j)$ and
$C^\bullet_\unr(G_{F_v},\scrT_j)$, $v \in S \bs (S_p \cup S_\infty)$,
are acyclic and (Proposition~\ref{P:selmer complex}(2)) the $\wt
C^\bullet_f(F,\scrT_j)$ are independent of $S$ and $\Sigma$.  Because
we may take $\Sigma$ to be the set of ramified places for $\scrT$, it
then follows immediately from general base-change results that the
natural map is a quasi-isomorphism for $j > 0$,
\[
\widetilde C^\bullet_f(F,\scrT_j) \Lotimes_{A_j} A_{j-1}
\stackrel\sim\to
\widetilde C^\bullet_f(F,\scrT_{j-1}).
\]
In particular, one has short exact sequences
\[
0 \to \widetilde H^i_f(F,\scrT_j)/T_j
\to
\widetilde H^i_f(F,\scrT_{j-1})
\to
\widetilde H^{i+1}_f(F,\scrT_j)[T_j] \to 0,
\]
where $(-)[T_j]$ denotes the subset of elements killed by $T_j$.
If $j > 0$ then the localization $A_{j,(T_j)}$ is a discrete valuation
ring with uniformizer $T_j$, fraction field $\Frac(A_j)$, and residue
field $\Frac(A_{j-1})$.  We thus obtain short exact sequences
\begin{equation}
\label{E:grothendieck duality}
0 \to \widetilde H^i_f(F,\scrT_j)_{(T_j)}/T_j
\to
\widetilde H^i_f(F,\scrT_{j-1}) \otimes_{A_{j-1}} \Frac(A_{j-1})
\to
\widetilde H^{i+1}_f(F,\scrT_j)_{(T_j)}[T_j] \to 0.
\end{equation}

One may now apply the arguments of \cite[Proposition~5.2.2 parts (4)
  through (7)]{N3}, using the above short exact sequence in place of
part (3) there, over each of the discrete valuation rings
$A_{j,(T_j)}$, $j=1,\ldots,r$, to show the following in turn:
\begin{itemize}
\item[(a)] Applying \cite[Theorem~10.2.3 and Proposition 10.2.5]{N-selmer} to \eqref{E:grothendieck duality} for $i=1$, there exists a nondegenerate skew-symmmetric pairing
\[
\widetilde H^2_f(F,\scrT_j)_{(T_j),A_{j,(T_j)}\text{-tors}}
\times
\widetilde H^2_f(F,\scrT_j)_{(T_j),A_{j,(T_j)}\text{-tors}}
\to (\Frac A_j)/A_{j,(T_j)}.
\]
\item[(b)] By (a) and the structure of symplectic modules over a DVR, there exists a finite-length $A_{j,(T_j)}$-module $Z_j$ such
  that $\widetilde H^2_f(F,\scrT_j)_{(T_j),A_{j,(T_j)}\text{-tors}}
  \cong Z_j \oplus Z_j$.
\item[(c)] By looking at \eqref{E:grothendieck duality} for $i=0$, the $A_{j,(T_j)}$-module $\widetilde
  H^1_f(F,\scrT_j)_{(T_j)}$ is torsion-free and hence free of rank $\widetilde
  h^1_f(F,\scrT_j)$.
\item[(d)] By (b) and \eqref{E:grothendieck duality} for $i=1$, one has $\widetilde h^1_f(F,\scrT_j) \equiv \widetilde
  h^1_f(F,\scrT_{j-1}) \pmod2$.
\end{itemize}
One then assembles the congruences
\[
\widetilde h^1_f(F,V)
=
\widetilde h^1_f(F,\scrT_0)
\equiv
\widetilde h^1_f(F,\scrT_1)
\equiv
\cdots
\equiv
\widetilde h^1_f(F,\scrT_r)
=
\widetilde h^1_f(F,\scrT) \pmod2
\]
to complete the proof of the proposition.
\end{proof}

\subsection{Global analytic invariants}

In this subsection we treat the variation of global $\vep$-factors in
families.  We continue with the notations and hypotheses of
Subsection~\ref{SS:global data}.

Let $(L,V,j,\{S_v\}_{v \in S_p})$ be a global point datum.  To every
finite place $v$ one can attach a local $\vep$-factor $\vep_v(V) =
\vep(WD(V_v)) \in \{\pm1\}$ (when $v \in S_p$, recall that $V_v$ is
Panchishkin and hence de~Rham).  All $v$ where $V$ is unramified, and
in particular all $v \notin S$, satisfy $\vep_v(V) = 1$.  We do not
know how to compute the individual factors $\vep_v(V)$ for archimedean
places $v \in S_\infty$, but we can compute their product, namely
\begin{equation}
\label{E:epsilon at infinity}
\vep_\infty(V) = (-1)^{r_2(F)\dim_L V/2} (-1)^{d^-(V)},
\end{equation}
where $r_2(F)$ is the number of pairs of complex embeddings of $F$
and, as in Subsection~\ref{S:local1},
\[
d^-(V_v): = \sum_{i<0} i \dim_L \Gr^i \bbD_\dR(V_v)
\textrm{ for } v \in S_p
\qquad \text{and} \qquad
d^-(V) := \sum_{v \in S_p} d^-(V_v).
\]

\begin{remark}\label{R:gap}
Let $M$ be a motive over $F$ with coefficients in a number field $E$,
symplectic self-dual and pure (of weight $-1$), let $\gothp$ be a
prime of $E$ lying over $p$, and suppose $L = E_\gothp$ and $V$ is the
$\gothp$-adic \'etale realization of $M$.  By classifying the possible
Hodge structures attached to $M$ and applying
\cite[\S\S5.2--5.3]{Del}, one can compute that for each embedding
$\iota \cn E \hookrightarrow \CC$, the product of archimedean
$\vep$-factors of $M$ (with respect to $\iota$) is given by
\[
\vep_\infty(\iota M)
=
\prod_{u \in S_\infty} \vep_u(\iota M)
=
(-1)^{r_2(F)\rank M/2}
\cdot
(-1)^{\sum_{u \in S_\infty,i<0} [F_u:\RR] i d^i_u(\iota M)},
\]
where $d^i_u(\iota M) = \dim_\CC \Gr^i M_\dR \otimes_{E \otimes_\QQ F,
  \iota \otimes \tilde u} \CC$ for any embedding $\tilde u \cn F
\hookrightarrow \CC$ giving rise to $u$.  We claim that
$\vep_\infty(\iota M) = \vep_\infty(V)$.  In fact, by inspection one
has
\[
\vep_\infty(\iota M) = (-1)^{r_2(F)\rank M/2} \vep_\infty(\iota \Ind^F_\QQ M)
\quad\text{and}\quad
\vep_\infty(V) = (-1)^{r_2(F)\dim_L V/2} \vep_\infty(\Ind^F_\QQ V),
\]
so it suffices to treat the case where $F=\QQ$.  Then $S_\infty =
\{\infty\}$ and $S_p = \{p\}$, and one has
\begin{gather*}
d^i_\infty(\iota M)
=
\dim_\CC \Gr^i M_\dR \otimes_{E,\iota} \CC
=
\dim_E \Gr^i M_\dR
=
\dim_L \Gr^i M_\dR \otimes_E L, \\
\Gr^i M_\dR \otimes_E L
\cong
\Gr^i \bbD_\dR(V_p),
\end{gather*}
which imply the desired claim.  (Another argument to show that
$\vep_\infty(\iota M) = \vep_\infty(V)$ is given in \cite{N3e}, and
this argument gives more information than ours, but we warn that some
of its steps are invalid when $F$ is not totally real.)

\end{remark}

We define the global $\vep$-factor as the product of the local
$\vep$-factors:
\[
\vep(V)
=
\vep_\infty(V) \cdot \prod_{v \notin S_\infty} \vep_v(V)
\in
\{\pm1\}.
\]
Applying Proposition~\ref{P:panchishkin epsilon} for each $v \in S_p$
gives 
\[
\vep(V)
=
(-1)^{\sum_{v \in S_p} h^0(D_v^-)}
(-1)^{r_2(F)\dim_LV/2}
\prod_{v \in S_p} (\det(D_v^+))(-1)
\cdot
\prod_{v \notin S_p \cup S_\infty} \vep_v(V).
\]
We define the modified $\vep$-factor to be
\begin{align}\label{E:modified epsilon}
\tilde\vep(V)
&=
(-1)^{\sum_{v \in S_p} h^0(D_v^-)} \vep(V) \\
&=
(-1)^{r_2(F)\dim_LV/2}
\prod_{v \in S_p} (\det(D_v^+))(-1)
\cdot
\prod_{v \notin S_p \cup S_\infty} \vep_v(V). \nonumber
\end{align}

Let $\scrX = (X,\scrT,j,\{\scrS_v\}_{v \in S_p})$ be a global datum.
For each place $v \notin S_p \cup S_\infty$ we form $\vep_v(\scrT) =
\vep(WD(\scrT_v) \otimes_{\calO_X} \kappa(X))$, a rational function on
$X$ with values in $\{\pm1\}$.  Then we set
\[
\tilde\vep(\scrT)
=
(-1)^{r_2(F)\rank_{\calO_X}\scrT/2}
\prod_{v \in S_p} (\det(\scrD_v^+))(-1)
\cdot
\prod_{v \notin S_p \cup S_\infty} \vep_v(\scrT),
\]
again a rational function on $X$ with values in $\{\pm1\}$.

\begin{prop}\label{P:vep_an}
If $\scrX$ is a global datum with $X$ irreducible and $P \in X_\alg$,
then $\tilde\vep(V) = \tilde\vep(\scrT)$.
\end{prop}

\begin{proof}
It suffices to compare $\tilde\vep(V)$ and $\tilde\vep(\scrT)$
factor-by-factor.  The first two factors are obviously unchanged by
specialization from $X$ to $P$, and for the final factor one invokes
Proposition~\ref{P:local eps locally constant}.
\end{proof}

\subsection{Main result}

We continue with the notations and hypotheses of
Subsection~\ref{SS:global data}, as well as the two intervening
subsections.

Suppose $(L,V,j,\{S_v\}_{v \in S_p})$ is a global point datum.  Then
the by \emph{parity conjecture} we mean the claim that the sign
$\vep(V)/(-1)^{h^1_f(F,V)} \in \{\pm1\}$ is equal to $1$.  When $V$ is
semisimple, the Fontaine--Mazur conjecture predicts that $V$ is the
$p$-adic \'etale realization of a motive, and when $V$ is motivic in
this sense the Bloch--Kato conjecture would imply the claim of the
parity conjecture.  Although we make no predictions on the validity of
the claim if $V$ is not motivic, the results of this paper apply to a
general global point datum regardless of whether $V$ is motivic.

Our main result is the following.  It generalizes
\cite[Theorem~5.3.1]{N3}, as per Example~\ref{Ex:nekovar}.

\begin{theorem}\label{T:main}
Let $\scrX$ be a global datum with $X$ irreducible.  Then one has
\[
\vep(V)/(-1)^{h^1_f(F,V)}
=
\tilde\vep(\scrT)/(-1)^{\tilde h^1_f(F,\scrT)},
\]
and in particular the sign $\vep(V)/(-1)^{h^1_f(F,V)}$ is independent
of $P \in X_\alg$.
\end{theorem}

\begin{proof}
One assembles the equalities
\[
\vep(V)/(-1)^{h^1_f(F,V)}
=
\tilde\vep(V)/(-1)^{\tilde h^1_f(F,V)}
=
\tilde\vep(\scrT)/(-1)^{\tilde h^1_f(F,\scrT)},
\]
the first by \eqref{E:modified epsilon} and Proposition~\ref{P:selmer
  complex}(2), and the second by Proposition~\ref{P:vep_an} and
Proposition~\ref{P:vep_alg}.
\end{proof}

\numberwithin{theorem}{section}
\section{Hilbert modular forms}\label{S:Hilbert}

The previous section is concerned with relating the parity conjecture
of two different objects, when they are members of the same family.
By contrast, the present section gives a supply of objects where the
parity conjecture is actually known.

In numerous works, Nekov\'a\v{r} has presented techniques allowing one
to prove the parity conjecture for a wide class of Hilbert modular
forms, subject most significantly to the condition that they be
ordinary (up to twist) at places above $p$; see for example
\cite{NGr}, where many of the techniques are employed.  We gather
these techniques into one place, and use the main result of the
preceding section to weaken the condition to being finite-slope (up to
twist) at places above $p$.

Fix a prime $p \neq 2$.

We fix $F$ a totally real field and an algebraic closure $F^\alg/F$.
In contrast to the previous sections, here the letter $K$ will always
denote a finite extension of $F$ inside $F^\alg$.  When $S$ is a
finite set of primes of $F$, we let $K_S$ be the maximal extension of
$K$ inside $F^\alg$ that is unramified outside places lying over those
in $S \cup \{\infty\}$, and we set $G_{K,S} = \Gal(K_S/K)$.  For
brevity, if $I \subseteq \calO_F$ is a nonzero ideal, we write
$G_{K,I}$ for $G_{K,S}$ with $S$ taken to be the set of primes
dividing $I$.

We consider cuspidal Hilbert newforms $f$ over $F$ of parallel weight
$2$, trivial central character $\om_f = 1$, level $N = N_f$, and
coefficients embedded in a finite extension of $\Qp$.  If $f$ has CM,
we let its field of CM be $K'/F$.  There is associated to $f$ a
two-dimensional continuous $G_{F,N}$-representation $V = V_f$.  We
assume $f$ is \emph{finite-slope (up to twist) at $p$}, which means
that for every prime $v$ of $F$ lying over $p$, the local constituent
at $v$ of its associated automorphic representation is not
supercuspidal; this is equivalent to $\bbD_\rig(V_f|_{G_{F_v}})$ being
Panchishkin.

We consider quadratic CM extensions $K/F$ of discriminant $D = D_K$
and quadratic character $\eta = \eta_K$ over $F$, the latter
considered interchangeably as either a homomorphism $G_{F,D} \to \{\pm1\}$ or
as a homomorphism $\bfA_F^\times/F^\times \to \{\pm1\}$.  Given also $f$, we
write $f_\eta$ for the newform associated to the twist $f \otimes
\eta$, so that $V_{f_\eta} = V_f \otimes \eta$.

Given $K$ and a prime $P$ of $F$ lying over $p$, we let $\wt G_n$
denote the ring class group of conductor $P^n$, defined by
$\bfA_K^\times/(K^\times K_\infty^\times \wh\calO_{P^n}^\times)$ where
$K_\infty = K \otimes_\QQ \RR$ and $\calO_{P^n} = \calO_F+P^n\calO_K$.
Let $\wt G_\infty = \varprojlim_n \wt G_n$, and let $\Delta$ be its
(finite) torsion subgroup.  Let $G_\infty = \wt G_\infty/\Delta$, so
that $G_\infty$ is noncanonically isomorphic to $\Zp^{[F_P:\Qp]}$, and
let $G_n$ denote the quotient of $\wt G_n$ by the image of $\Delta$.
When $K$ and $P$ are understood, we consider finite-order continuous
characters $\chi$ of $G_\infty$.  The \emph{conductor} of $\chi$ is by
definition $P^n$, for the smallest integer $n \geq 0$ such that $\chi$
factors through $G_n$.  Note that the image of $\bfA_F^\times$ in $\wt
G_\infty$ is contained in $\Delta$, so that $\chi|_{\bfA_F^\times} =
1$.  Note also that since $p$ is odd, another quadratic extension
$K'/F$ has $K' \neq K$ if and only if $K' \nsubseteq (F^\alg)^{\ker
  \chi}$.

With $f$, $K$, and $\chi$ as above, there are the Bloch--Kato Selmer
groups
\[
H^1_f(F,V),\ H^1_f(F,V_{f_\eta}),\ H^1_f(K,V_f\otimes\chi),\ 
H^1_f(K,V_f) \cong H^1_f(F,V_f) \oplus H^1_f(F,V_{f_\eta}),
\]
of respective dimensions
\[
h^1_f(f),\ h^1_f(f_\eta),\ h^1_f(f,\chi),\
h^1_f(f/K) = h^1_f(f) + h^1_f(f_\eta).
\]
On the other hand, we write
\[
\vep(f),\ \vep(f_\eta),\ \vep(f,\chi),\
\vep(f/K) = \vep(f)\vep(f_\eta),
\]
respectively, for the sign of the functional equation of
\[
L(f,s),\ L(f_\eta,s),\ L(f,\chi,s),\
L(f,1,s) = L(f,s)L(f_\eta,s),
\]
where the latter two are Rankin--Selberg $L$-functions.  There are
also local signs $\vep_v(?)$, and the product formulas $\vep(?)  =
\prod_v \vep_v(?)$.  For each $? \in \{f, f_\eta, (f,\chi), f/K\}$ the
\emph{parity conjecture} for $?$ is the claim that $(-1)^{h^1_f(?)} =
\vep(?)$.  If the parity conjecture holds for $f/K$, then the parity
conjecture for $f$ is equivalent to the parity conjecture for
$f_\eta$.  These definitions all agree with the setup in the previous
section; in particular, the product of these epsilon factors at
infinity is given by \eqref{E:epsilon at infinity}.

The following result is our main contribution to the circle of ideas
of this section.

\begin{theorem}\label{T:parity-twisting}
Given any $f,K,P$ as above, the parity conjecture for $f/K$ is
equivalent to the parity conjecture for $(f,\chi)$ for any $\chi$ as
above.
\end{theorem}

\begin{proof}
This follows from the finite-slope hypothesis and
Theorem~\ref{T:main}, using the family of twists of $V|_{G_{K,ND}}$ by
continuous characters of $G_\infty$.  We only need to verify that
$V|_{G_{K,ND}}$ and $V|_{G_{K,ND}} \otimes \chi$ are pure.  Since
$\chi$ has finite order, it suffices to check for $V|_{G_{K,ND}}$, and
for the latter it suffices to know purity for $V$.  This fact follows
from combining the Ramanujan conjecture and local-global
compatibility for Hilbert modular forms, which were shown in
increasing generality by many people, the final cases being treated in
\cite[Theorem~1]{Bl} and \cite[Theorem~1]{S}, respectively.
\end{proof}

To treat the parity conjecture for $f$, we will momentarily divide
into two cases.  The precise inputs needed depend on which case we are
in, but the proof generally proceeds in both cases by the following
two steps, which first appeared in \cite{N2}.  \emph{Step 1} is to
verify the parity conjecture for $f/K$ for some class of $K$.  First,
one invokes nonvanishing results for Rankin--Selberg $L$-values or CM
points on Jacobians of Shimura curves, which involve a possibly
high-conductor twist $\chi$.  Second, one uses Euler system machinery
to convert these nonvanishing results into bounds on Selmer groups, to
obtain the parity conjecture for $(f,\chi)$.  Finally, one uses the
above theorem to get rid of $\chi$.  \emph{Step 2} is to use
nonvanishing theorems for quadratic-twisted $L$-values to choose a $K$
as in Step 1 where the Euler system method is powerful enough to give
the parity conjecture for $f_\eta$ itself.

We say that $f$ \emph{admits an indefinite case} if it satisfies the
equivalent conditions of \cite[Proposition~2.10.2]{NCan}, one of which
being that $[F:\QQ]$ is odd or $f$ is not principal series at some
finite place.  Note that this condition is invariant under replacing
$f$ by a quadratic twist.

The reader may compare the following theorem to
\cite[Theorem~1.4]{NANT}.

\begin{theorem}\label{T:indefinite}
Recall that we have assumed $p$ is odd, and $f$ has trivial central
character and is finite-slope (up to twist) at $p$.  If $f$ admits an
indefinite case, then the parity conjecture holds for $f$.
\end{theorem}

\begin{proof}
\emph{Step 1.}  We show the following claim: Assume that $K$ satisfies
$\vep(f/K)=-1$, and that if $f$ has CM then $K' \neq K$.  Then the
parity conjecture holds for $f/K$.

Fix an archimedean place $\tau$ of $F$.  Let $\Sigma$ be the set of
places $v$ of $F$ not equal to $\tau$ such that
$\eta_v(-1)\vep_v(f/K)=-1$.  All infinite places $v$ of $F$ not equal
to $\tau$ belong to $\Sigma$, because $\eta_v(-1)=-1$ and
$\vep_v(f/K)=1$.  At each finite $v \in \Sigma$, $K/F$ is not split
and $f$ is special or supercuspidal.  There exists a unique quaternion
algebra $B/F$ that has nonsplit set $\Sigma$.  The Jacquet--Langlands
correspondence associates $f$ to a cuspidal automorphic representation
$\pi'$ of $B^\times$ whose vectors are differentials on a system of
$B$-Shimura curves.  We fix an $F$-algebra embedding $t \cn K
\hookrightarrow B$, allowing us to form $K$-CM points on these Shimura
curves.

We apply \cite[Theorem~2.5.1]{AN}; let us specify how our situation
fits into the notations and hypotheses there.  We use the quaternion
algebra $B$ above.  In the data (2.4.1) of {\it loc. cit.}, we use (D1) $s=1$, $P_1 = P$,
(D2) $c=1$, (D3) the class $\calC$ is nontrivial but otherwise
arbitrary, (D4) the quotient $A_0$ is to be specified momentarily, and
(D5) $\chi_0=1$ (because our $\chi=1$).  In (2.4.3) of {\it loc. cit.} the set $S$
consists of those primes distinct from $P$ that ramify in $K/F$, and
$\chi_1=1$ because $\chi_0=1$.  It remains to choose $A_0$ among those
in the Shimura curve Jacobians corresponding to $\pi'$ so that the
equivalent conditions of Proposition 2.4.10(1) of {\it loc. cit.} are satisfied.  But
$c=1$ implies all primes in $S$ ramify in $K/F$, and in particular do
not split in $K/F$, so the discussion of (2.4.11) of {\it loc. cit.} produces a nonzero
${}^\sigma\ell$: if the vector $v$ satisfying ${}^\sigma\ell(v) \neq
0$ is fixed by the level $H$, then some quotient $A_0$ of the Shimura
curve Jacobian of level $H$ will do.  In summary, there exist $K$-CM
points on $A_0$ of arbitrarily high conductor that are nontorsion.

Then we apply \cite[Theorem~3.2]{NDur}, to $A_0$, $S_B = \Sigma$, a
nontorsion $K$-CM point guaranteed by the preceding paragraph, and any
character $\alpha$ whose component of the $K$-CM point is nontorsion
(note that $\alpha$ is necessarily a finite-order character of
$G_\infty$).  Therefore one has $(-1)^{h^1_f(f,\alpha)}=-1$, which,
combined with \cite[Proposition~2.6.2(2)]{AN} with $\alpha$ in place
of $\chi$, establishes the parity conjecture for $(f,\alpha)$.  We now
appeal to Theorem~\ref{T:parity-twisting} to obtain the claim.

\emph{Step 2.}  \cite[Theorem~B]{FH} gives infinitely many $K_1,K_2$
(distinct from $K'$ if $f$ has CM) with
\[
\vep(f/K_1) = \vep(f_{\eta_{K_1}}/K_2) = -1,
\quad
\{\ord_{s=1} L(f_{\eta_{K_1}},s),
  \ord_{s=1} L((f_{\eta_{K_1}})_{\eta_{K_2}},s)\}
=
\{0,1\}.
\]
Applying Step 1 twice shows that the parity conjectures for $f$,
$f_{\eta_{K_1}}$, and $(f_{\eta_{K_1}})_{\eta_{K_2}}$ are all
equivalent.  But letting $g \in \{f_{\eta_{K_1}},
(f_{\eta_{K_1}})_{\eta_{K_2}}\}$ be such that $\ord_{s=1} L(g,s)=0$,
it follows from \cite[Corollary~2.10.4]{NCan} that $h^1_f(g)=0$, so
that the parity conjecture holds for $g$.
\end{proof}

If $f$ does not admit an indefinite case, so that $\vep(f/K) = +1$ for
all $K$, more hypotheses are necessary.  The following condition (A2)
for $f$ arises in \cite{NCan}.
\begin{itemize}
\item[(A2)] there exists $g \in G_{F,N}$ such that $\det(1-gX|V) =
  (1-\lambda_1X)(1-\lambda_2X)$ with $\lambda_1^2 = 1 \neq
  \lambda_2^2$, and if $f$ has CM then moreover $\lambda_2^n \neq 1$
  for all $n \geq 1$.
\end{itemize}
Note that condition (A2) is invariant under replacing $f$ by a
quadratic twist.

The reader may compare the following theorem to
\cite[Theorem~3.5]{NANT} when $f$ has irreducible residual
representation.

\begin{theorem}\label{T:definite}
Recall we have assumed $p$ is odd, and $f$ has trivial central
character and is finite-slope (up to twist) at $p$.  If $f$ does not
admit an indefinite case, condition (A2) holds, and $\vep(f)=+1$, then
the parity conjecture holds for $f$.
\end{theorem}

\begin{proof}
\emph{Step 1.}  We show the following claim: Assume that $D$ is
coprime to $N$, that $P$ is split in $K/F$, and that if $f$ has CM
then $K \neq K'$. Then the parity conjecture holds for $f/K$.

Because $P$ is split in $K/F$, for all $\chi$ the sign $\vep(f,\chi) =
\vep(f/K) = 1$ is constant, and by \cite[Theorem~1.4]{CV}, one can
find $\chi$ of sufficiently large $P$-power conductor such that
$L(f,\chi,1) \neq 0$.  Then \cite[Theorem~A]{NCan} shows that
$h^1_f(f,\chi) = 0$.  Thus the parity conjecture holds for $(f,\chi)$,
hence also for $f/K$ by Theorem~\ref{T:parity-twisting}.

\emph{Step 2.}  Fix two primes $Q_1,Q_2$ of $F$ not dividing $NP$.
\cite[Theorem B]{FH} gives infinitely many $K$ (distinct
from $K'$ if $f$ has CM) such that every prime dividing $NP$ is split
in $K/F$, the primes $Q_1,Q_2$ ramify in $K/F$, and $L(f_\eta,1) \neq
0$.  The proof of \cite[Theorem~B]{NCan} shows that $h^1_f(f_\eta) =
0$, so that the parity conjecture holds for $f_\eta$.  But Step 1
shows that the parity conjecture for $f_\eta$ is equivalent to the
parity conjecture for $f$.
\end{proof}

\end{document}